\documentclass[12pt]{amsart}

\newtheorem{theorem}{Theorem}[section]
\newtheorem{lemma}[theorem]{Lemma}

\newtheorem{proposition}[theorem]{Proposition}

\newtheorem{corollary}[theorem]{Corollary}

\usepackage[bookmarks]{hyperref}
\usepackage{cite}
\usepackage{graphicx,amscd,amsthm,amsfonts,amsopn,amssymb,verbatim,enumerate,xcolor}

\usepackage[letterpaper,left=1in,top=1in,right=1in,bottom=1in]{geometry}

\def\half{ \frac{1}{2}}
\def\D{\partial}

\def\R{{\mathbb R}}

\def\nint{\mathop{\diagup\kern-13.0pt\int}}
\def\Z{{\mathbb Z}}

\def\bas{\begin{align*}}
\def\eas{\end{align*}}
\def\bi{\begin{itemize}}
\def\ei{\end{itemize}}

\def\emph#1{{\it #1}}

\def\eps{{\epsilon}}

\def\dq{{\delta}}
\def\sdq{{|\delta|}}

\def\pax{\mathbf{u}}
\def\rpax{\mathbf{r}}

\def\PVy{{\text{p.v.} |y|^{-1}}}

\DeclareMathOperator{\sgn}{sgn}

\theoremstyle{definition}

\numberwithin{equation}{section}

\begin{document}

\title{Well-posedness for the surface quasi-geostrophic front equation}
\author{Albert Ai}
\address{Department of Mathematics, University of Wisconsin, Madison}
\email{aai@math.wisc.edu}
\author{Ovidiu-Neculai Avadanei}
\address{Department of Mathematics, University of California at Berkeley}
\email{ovidiu\_avadanei@berkeley.edu}
\keywords{SQG front equation, paralinearization, modified energies, frequency envelopes, wave packet testing}
\subjclass{35Q35}
\begin{abstract}
We consider the well-posedness of the surface quasi-geostrophic (SQG) front equation. Hunter-Shu-Zhang  \cite{HSZglobal} established well-posedness under a small data condition as well as a convergence condition on an expansion of the equation's nonlinearity. In the present article, we establish unconditional large data local well-posedness of the SQG front equation, while also improving the low regularity threshold for the initial data. In addition, we establish global well-posedness theory in the rough data regime by using the testing by wave packet approach of Ifrim-Tataru.
\end{abstract}

\maketitle
\addtocontents{toc}{\protect\setcounter{tocdepth}{1}}
\tableofcontents

\section{Introduction}

The surface quasi-geostrophic (SQG) equation takes the form 
\begin{equation}\label{rawSQG}
\theta_t +  u \cdot \nabla \theta = 0, \qquad u = (- \Delta )^{-\half} \nabla^\perp \theta
\end{equation}
where $\theta$ is a scalar evolution equation on $\R^2$, $(-\Delta)^{-\half}$ denotes a fractional Laplacian, and $\nabla^\perp = (-\D_y, \D_x)$. The SQG equation arises from oceanic and atmospheric science as a model for quasi-geostrophic flows confined to a surface. This equation is also of interest due to similarities with the three dimensional incompressible Euler equation. In particular, the question of singularity formation remains open for both problems.

The SQG equation is one member in a family of two-dimensional active scalar equations parameterized by the transport term in \eqref{rawSQG}, with 
\begin{equation}\label{gSQG}
u = (- \Delta )^{-\frac{\alpha}{2}} \nabla^\perp \theta, \qquad \alpha \in (0, 2].
\end{equation}
The case $\alpha = 2$ gives the two dimensional incompressible Euler equation, while the $\alpha = 1$ case gives the SQG equation \eqref{rawSQG} above.

\

Front solutions to \eqref{rawSQG} refer to piecewise constant solutions taking the form
\[
\theta(t, x, y) = \begin{cases}
  \theta_+  \qquad \text{if } y > \varphi(t, x) \\
  \theta_-  \qquad \text{if } y < \varphi(t, x)
\end{cases},
\]
where the front is modeled by the graph $y = \varphi(t, x)$ with $x \in \R$. Front solutions are closely related with patch solutions
\[
\theta(t, x, y) = \begin{cases}
  \theta_+  \qquad \text{if } (x, y) \in \Omega(t) \\
  \theta_-  \qquad \text{if } (x, y) \notin \Omega(t)
\end{cases},
\]
where $\Omega$ is a bounded, simply connected domain. 

\
  In the SQG case $\alpha = 1$, the evolution equation for the front $\varphi$, derived in \cite{HSderivation, HSZderivation}, takes the form 
\begin{equation}\label{SQG}
\begin{aligned}
\partial_t\varphi(t,x) - A_\varphi \varphi_x(t, x) &= 2\log|D_x|\D_x\varphi(t,x), \\
\varphi(0,x)&=\varphi_0(x)
\end{aligned}
\end{equation}
where $\varphi$ is a real-valued function $\varphi : [0, \infty) \times \R \rightarrow \R$ and
\begin{equation}\label{A-op}
A_\varphi \varphi_x(t, x) = \int \left(\frac{1}{|y|} - \frac{1}{\sqrt{y^2+(\varphi(t,x+y)-\varphi(t,x))^2}}\right)\cdot (\varphi_x(t,x + y) - \varphi_x(t,x)) \,dy.
\end{equation}
The equation \eqref{SQG} is invariant under the transformation
\begin{align*}
    t\rightarrow\kappa t,\qquad x\rightarrow\kappa(x+\log|\kappa|t),\qquad\varphi\rightarrow\kappa\varphi,
\end{align*}
which means that $\dot{H}^{\frac{3}{2}}(\mathbb{R})$ is the corresponding critical Sobolev space.

\

When $\alpha \in (0, 1)$, contour dynamics equations for patches and fronts may be derived and analyzed in a similar way. Global well-posedness for small and localized data was established by Córdoba-Gómez-Serrano-Ionescu in \cite{GlobalPatch}.

However, when $\alpha \in [1, 2]$, the derivation of contour dynamics equations for fronts has complexities arising from the slow decay of Green's functions. The derivation in this range was addressed by Hunter-Shu \cite{HSderivation} via a regularization procedure, and again by Hunter-Shu-Zhang in \cite{HSZderivation}.

In the case of SQG patches, Gancedo-Nguyen-Patel proved in \cite{GNP} that in a suitable parametrization, the SQG front equation is locally well-posed in $H^s(\mathbb{T})$, where $s>2$. Local well-posedness for the generalized SQG family, where $\alpha\in(0,2)$ and $\alpha \neq 1$, was also considered by Gancedo-Patel in \cite{GP}, establishing in particular local well-posedness in $H^2$ for $\alpha \in (0, 1)$. 

In the case of nonperiodic SQG fronts, Hunter-Shu-Zhang studied the local well-posedness for a cubic approximation of \eqref{SQG} in \cite{HSZapprox}. In \cite{HSZglobal}, they considered the local well-posedness for the full equation \eqref{SQG} with initial data in $H^s$, $s \geq 5$, along with global well-posedness for small, localized, and essentially smooth ($s \geq 1200$) initial data. However, these well-posedness results require a small data assumption to ensure the coercivity of the modified energies used in the energy estimates, along with a convergence condition on an expansion of the nonlinearity $A_\varphi \varphi$ appearing in \eqref{SQG}. These results were extended to the range $\alpha \in (1, 2]$ for the generalized SQG family in \cite{HSZfamily}. 

In the present article, our objective is to revisit and streamline the analysis of \eqref{SQG}, while improving the established well-posedness results. Our contributions include:
\begin{itemize}
\item  offering substantial simplifications to the paradifferential analysis,
\item removing the convergence and small data assumptions in the local well-posedness result of \cite{HSZglobal},
\item establishing the local well-posedness in a significantly lower regularity setting at $1+\epsilon$ derivatives above scaling, corresponding to the classical threshold of Hughes-Kato-Marsden for nonlinear hyperbolic systems \cite{HKM}, and
\item establishing the global well-posedness in a low regularity setting, by applying the wave packet testing method of Ifrim-Tataru (see for instance \cite{ITschrodinger, ITpax}).
\end{itemize}
We anticipate that our streamlined analysis will also open the way to substantial simplifications and improvements in the analysis of related equations, including the generalized SQG family \eqref{gSQG}.

\

Our main local well-posedness result is as follows:
\begin{theorem}\label{t:lwp}
Equation \eqref{SQG} is locally well-posed for initial data in $H^s$ with $s > \frac{5}{2}$.  Precisely, for every $R > 0$, there exists $T=T(R)>0$ such that for any $\varphi_0\in H^s(\R)$ with $\|\varphi_0\|_{H^s} < R$, the Cauchy problem \eqref{SQG} has a unique solution $\varphi \in C([0, T], H^s)$. Moreover,  the solution map $\varphi_0 \mapsto \varphi$ from $H^s$ to $C([0, T], H^s)$ is continuous.
\end{theorem}

We also consider global well-posedness for small and localized data. To describe localized solutions, we define the operator 
\[
L = x+2t+2t\log|D_x|,
\]
which commutes with the linear flow $\partial_t-2\log|D_x|\D_x$, and at time $t = 0$ is simply multiplication by $x$. Then we define the time-dependent weighted energy space
\[
\|\varphi\|_X := \|\varphi\|_{H^s} + \|L\partial_x \varphi\|_{L^2},
\]
where $s>4$. To track the dispersive decay of solutions, we define the pointwise control norm 
\[
\|\varphi\|_Y := \||D_x|^{3/4-\delta}\varphi\|_{L_x^{\infty}} + \||D_x|^{2 + \delta}\varphi\|_{L_x^\infty}.
\]

\begin{theorem}\label{t:gwp}
Consider data $\varphi_0$ with
\[
\|\varphi_0\|_X \lesssim \eps \ll 1.
\]
Then the solution $\varphi$ to \eqref{SQG} with initial data $\varphi_0$ exists globally in time, with energy bounds 
\[
\|\varphi(t)\|_{X} \lesssim \eps t^{C\eps^2}
\]
and pointwise bounds 
\[
\|\varphi(t)\|_Y \lesssim \eps \langle t\rangle^{-\half}.
\]
\end{theorem}
 Further, the solution  $\varphi$ exhibits a modified scattering behavior, with an asymptotic profile $W$, in a sense that will be made precise in Section~\ref{s:scattering}. There, we also observe that \eqref{SQG} has a conserved $L^2$ mass. We remark that  Hunter-Shu-Zhang \cite{HSZfamily} makes a similar observation regarding mass conservation for the generalized SQG in the range $\alpha \in (1, 2]$.

\

Our paper is organized as follows. In Section~\ref{s:notation}, we establish notation and preliminaries used through the rest of the paper. We define a parameter-dependent paradifferential quantization which will ensure coercivity in our energy estimates, and which is key in removing the small data assumption in the local well-posedness theory. We also establish Moser estimates to be applied toward the paralinearization of $A_\varphi \varphi$, as well as a more general linearized counterpart $A_{\varphi}v$. Lastly, we record some elementary lemmas involving difference quotients.

In Section~\ref{s:paralinearization}, we paralinearize the operator $A_\varphi v$, up to a perturbative error. We then apply this result to reduce the analysis of both the equation \eqref{SQG} and its linearization 
\begin{equation}\label{SQG-lin}
    \partial_tv - \D_x A_\varphi v =2\log|D_x| \D_xv
\end{equation}
to the analysis of a paradifferential flow with perturbative source.

In Section~\ref{s:energy}, we prove energy estimates for the paradifferential flow of Section \ref{s:paralinearization}, by defining modified energies that are comparable to the classical Sobolev energies in the spaces $H^s(\mathbb{R})$. This part crucially uses the quantization established in Section~\ref{s:notation} to remove the small data assumption made in the local well-posedness result from \cite{HSZglobal}.

In Section~\ref{s:lwp}, we prove Theorem~\ref{t:lwp}, the local well-posedness result for \eqref{SQG}. We do this by first using an iterative scheme to construct smooth solutions, and then using the method of frequency envelopes to lower the regularity exponent for the solutions. This method was introduced by Tao in \cite{25} in order to better track the evolution of energy distribution between dyadic frequencies. A systematic presentation of the use of frequency envelopes in the study of local well-posedness theory for quasilinear problems can be found in the expository paper \cite{ITprimer}. 

In Section~\ref{s:gwp} we use the wave packet testing method of Ifrim-Tataru to prove the global-wellposedness part of Theorem \ref{t:gwp}, along with the dispersive bounds for the resulting solution. This method is systematically presented in \cite{ITpax}. 

Finally, in Section~\ref{s:scattering} we provide a short proof for the mass conservation for the solutions of \eqref{SQG}. Then we discuss the modified scattering behavior of the global solutions constructed in Section \ref{s:gwp}.

\subsection{Acknowledgements} The first author was supported by the NSF grant DMS-2220519 and the RTG in Analysis and Partial Differential equations grant DMS-2037851. The second author was supported by the NSF
grant DMS-2054975, as well as by the Simons Foundation. The authors would like to thank Mihaela Ifrim and Daniel Tataru for many helpful discussions.

\section{Notation and preliminaries}\label{s:notation}

To facilitate the analysis of the operator $A_\varphi$, we define the smooth function 
\begin{equation}\label{F-def}
F(s) = 1 - \frac{1}{\sqrt{1+s^2}},
\end{equation}
which in particular vanishes to second order at $s = 0$, satisfying $F(0)=F'(0)=0$. Using this notation, we define the following generalization of \eqref{A-op}:
\begin{equation}\label{A-op-lin}
    (A_{\varphi}v)(t,x) = \int F(\dq^y\varphi(t,x)) \cdot \sdq^y v(t,x)\,dy.
\end{equation}
This generalized operator will be useful when we study the linearized equation.

We let $0 < \delta < s - \frac52$ refer to a small positive exponent throughout. Implicit constants may depend on our choice of $\delta$.

\subsection{Paradifferential operators and paraproducts}

Let $\chi$ be an even smooth function such that $\chi=1$ on $[-\frac{1}{20}, \frac{1}{20}]$ and $\chi = 0$ outside $[-\frac{1}{10}, \frac{1}{10}]$, and define
\[
\tilde{\chi}(\theta_1,\theta_2) = \chi\left(\frac{|\theta_1|^2}{M^2+|\theta_2|^2}\right).
\]
Given a symbol $a(x,\eta)$,  we use the above cutoff symbol $\tilde \chi$ to define an $M$ dependent paradifferential quantization of $a$ by (see also \cite{ABZgravity}) 
\begin{align*}
    \widehat{T_au}(\xi)=(2\pi)^{-1}\int \hat{P}_{>M}(\xi)\tilde \chi\left(\xi - \eta, \xi + \eta \right) \hat{a}(\xi-\eta,\eta)\hat{P}_{>M}(\eta)\hat{u}(\eta)\,d\eta,
\end{align*}
where the Fourier transform of the symbol $a=a(x,\eta)$ is taken with respect to the first argument. We make the following remarks:
\begin{itemize}
\item The Fourier transform, and in particular $T_a$ and Littlewood-Paley projections, will all operate with respect to the $x$ variable throughout. 
\item On the Fourier side, the support conditions on $\hat{P}_{>M}$ and $\tilde{\chi}$ imply that $|\eta|\approx |\xi|$. 
\item When $a$ is independent of $\eta$, the above definition coincides with the Weyl quantization, which means that the operator $T_a$ will be self-adjoint if $a$ is real valued. This will be useful when proving energy estimates in Section \ref{s:energy}.
 \end{itemize}

Implicit constants may depend on $M \gg 1$, which we view as a constant parameter except in Section~\ref{s:lwp}. There, we will choose $M$ using the following lemma, which will be key for establishing local well-posedness for large data:
\begin{lemma}\label{M-paraprod}
Let $R > 0$, $r\geq 1$, and $s > \half$. There exists $M$ such that $\|T_{1-(1-F(u))^r}\|_{L^2 \rightarrow L^2} < 1$ for any $u$ such that $\|u\|_{H^s} \leq R$.
\end{lemma}

\begin{proof}
On $|\xi-\eta| \leq \frac{M}{2}$, we have $ \tilde{\chi}(\xi-\eta,\xi + \eta)=1$. Thus we may write
 \begin{align*}
    2\pi \cdot \widehat{T_a f}(\xi) &= \int_{|\xi-\eta|\leq \frac{M}{2}}\hat{P}_{>M}(\xi)\hat{a}(\xi-\eta)\hat{P}_{>M}(\eta)\hat{f}(\eta)\,d\eta \\
    &\quad + \int_{|\xi-\eta|>\frac{M}{2}}\hat{P}_{>M}(\xi)\tilde{\chi}(\xi-\eta,\xi+\eta)\hat{a}(\xi-\eta)\hat{P}_{>M}(\eta)\hat{f}(\eta)\,d\eta
 \end{align*}
and conclude
 \begin{align*}
     \|T_a f\|_{L^2}  &\leq (\|P_{\leq \frac{M}{2}}a\|_{L^\infty} + C\|P_{> \frac{M}{2}}a\|_{L^\infty})\|f\|_{L^2}.
 \end{align*}
 
 We have by Sobolev embedding
 \[
 \|P_{> \frac{M}{2}}a\|_{L^\infty} \lesssim M^{-\frac{\delta}{2}}\|a\|_{H^s}
 \]
so that for $a = 1-(1-F(u))^r$ with $\|u\|_{H^s} \leq R$, we have 
 \[
 \|P_{> \frac{M}{2}}a\|_{L^\infty} \lesssim_R M^{-\frac{\delta}{2}}.
 \]
Further, using the form of $F$ combined with Sobolev embedding, we have $\displaystyle \|a\|_{L^\infty} < 1 - \delta$ uniformly over $\|u\|_{H^s} \leq R$, for some $\delta > 0$ depending on $R$. Choosing $M$ sufficiently large depending on $R$, we obtain 
\[
\|T_{1-(1-F(u))^r}\|_{L^2\rightarrow L^2} \leq \|1-(1-F(u))^r\|_{L^\infty} + \frac{\delta}{2} < 1 - \frac{\delta}{2}.
\]

 \end{proof}

\subsection{Classical estimates}

We recall the following Moser-Schauder estimates. Note that in the present article, we typically take $F$ to be given by \eqref{F-def}.

\begin{theorem}[Moser-Schauder]\label{t:moser}
Let $s \geq 0$ and $F:\R\rightarrow\R$ be a smooth function satisfying $F(0)=0$.

a) If $f\in H^s(\R)\cap L^\infty(\R)$,
\begin{equation}\label{moser}
    \|F(f)\|_{H^s}\lesssim_{F,\|f\|_{L^\infty}} \|f\|_{H^s}.
\end{equation}

b)  If $\partial_x^{-1}f\in H^s(\R)\cap L^\infty(\R)$ and $F'(0)=0$, we have
\begin{equation}\label{moser2}
\begin{aligned}
    \|F(f) - T_{F'(f)} f \|_{H^{s}} &\lesssim_{F,\|f\|_{W^{1,\infty}}}  \|f\|_{W^{1,\infty}}\|\partial_x^{-1}f\|_{H^s}, \\
\|F(f) - T_{F'(f)} f \|_{\dot H^{s}}&\lesssim_{F,\|f\|_{W^{1,\infty}}}\|f\|_{W^{1,\infty}} \|\partial_x^{-1}f\|_{\dot H^s}.
\end{aligned}
\end{equation}

c) Under the hypotheses of b), in the case $s = 0$, we have
\begin{equation}
 \|F(f) - T_{F'(f)} f \|_{L^2}\lesssim_F \|f\|_{W^{1,\infty}}\|\partial_x^{-1}f\|_{L^2}.  
\end{equation}
\end{theorem}

\begin{proof}
For a proof of a), see \cite[Lemma A.9]{DispersiveTao}. Here, we prove parts b) and c). In the following, denote
\[
f_{\leq_h k} := f_{\leq k - 1} + h f_k.
\]
Then write
\begin{equation*}\begin{aligned}
  F(f)&=\sum_{\substack{k=-\infty}}^\infty\int_0^1 F'(f_{\leq_h k})\cdot f_k\,dh \\
     &=\int_0^1\sum_{\substack{k=-\infty}}^\infty\left( F'(f_{\leq_h k})_{<k-4} + F'(f_{\leq_h k})_{\geq k - 4}\right)\cdot f_k \,dh\\
     &:=\int_0^1 B_1+B_2 \, dh.
\end{aligned}\end{equation*}

We consider $B_2$. Write
\begin{equation*}\begin{aligned}
\|B_2 \|_{\dot H^s}^2 &\lesssim\sum_{\substack{k=-\infty}}^\infty\left\|\sum_{\substack{j=-4}}^\infty F'(f_{\leq_h k - j})_k\cdot f_{k-j}\right\|_{\dot H^s}^2.
\end{aligned}\end{equation*}
We estimate the summand by
\begin{equation*}\begin{aligned}
  \left\|\sum_{\substack{j=-4}}^\infty F'(f_{\leq_h k - j})_k\cdot f_{k - j} \right\|_{\dot H^s}&\lesssim \sum_{\substack{j=-4}}^\infty \left\|| D_x|^s F'(f_{\leq_h k - j})_k\right\|_{L^\infty} \left\|f_{k - j}\right\|_{L^2}\\
  &\lesssim \sum_{\substack{j=-4}}^\infty 2^{sk-j}\left\|\left(F''(f_{\leq_h k - j})\cdot\D_x f_{\leq_h k - j}\right)_k\right\|_{L^\infty} \left\|\D_x^{-1}f_{k - j}\right\|_{L^2}.
\end{aligned}\end{equation*}
By the chain rule, for each $j \geq -4$,
\begin{equation*}\begin{aligned}
    \left\|\left(F''(f_{\leq_h k - j})\cdot\D_x f_{\leq_h k - j}\right)_k\right\|_{L^\infty}&\lesssim_{N} 2^{-Nj}K(\|f\|_{L^\infty},\|f_x\|_{L^\infty}).
\end{aligned}\end{equation*}
Thus,
\begin{equation*}\begin{aligned}
    \left\|\sum_{\substack{j=-4}}^\infty F'(f_{\leq_h k - j})_k\cdot f_{k - j} \right\|_{\dot H^s} &\lesssim K(\|f\|_{L^\infty},\|f_x\|_{L^\infty}) \sum_{\substack{j=-4}}^\infty 2^{(s-N-1)j}\left\|\D_x^{-1}f_{k-j}\right\|_{\dot H^s}.
\end{aligned}\end{equation*}
In this case of the frequencies $k \geq 0$ in $B_2$,
\begin{equation*}\begin{aligned}
    \sum_{\substack{k=0}}^\infty\left\|\sum_{\substack{j=-4}}^\infty F'(f_{\leq_h k - j})_k\cdot f_{k-j}\right\|_{\dot H^s}^2&\lesssim K(\|f\|_{L^\infty},\|f_x\|_{L^\infty})^2\sum_{k=0}^\infty  \left(\sum_{\substack{j=-4}}^\infty 2^{(s-N-1)j}\left\|\D_x^{-1}f_{k-j}\right\|_{\dot H^s}\right)^2\\
    &\lesssim K(\|f\|_{L^\infty},\|f_x\|_{L^\infty})^2\sum_{k=0}^\infty\left\|\D_x^{-1}f_k\right\|_{\dot H^s}^2\\
    &\lesssim K(\|f\|_{L^\infty},\|f_x\|_{L^\infty})^2\|\D_x^{-1}f\|_{\dot H^s}^2, 
\end{aligned}\end{equation*}
where $ N$ has been chosen such that $ N+1>s$.
In particular, when $s=0$, we may choose $N = 0$ to obtain
\begin{equation*}\begin{aligned}
     \sum_{\substack{k=0}}^\infty\left\|\sum_{\substack{j=-4}}^\infty F'(f_{\leq_h k - j})_k\cdot f_{k-j}\right\|_{L^2}^2\lesssim \|f_x\|_{L^\infty}^2\|\D_x^{-1}f\|_{L^2}^2.
\end{aligned}\end{equation*}
For the low frequency terms $k \leq 0$ in $B_2$, using the fact that $F'(0)=0$,
\begin{equation*}\begin{aligned}
    \sum_{\substack{k=-\infty}}^{-1}\left\|\sum_{\substack{j=-4}}^\infty F'(f_{\leq_h k - j})_k\cdot f_{k-j}\right\|_{\dot H^s}^2 &\lesssim \sum_{\substack{k=-\infty}}^{-1}\left\|\sum_{\substack{j=-4}}^\infty F'(f_{\leq_h k - j})_k\cdot f_{k-j}\right\|^2_{L^2}\\
    &\lesssim \sum_{\substack{k=-\infty}}^{-1}\left(\sum_{\substack{j=-4}}^\infty\left\|F'( f_{\leq_h k - j})_k\right\|_{L^\infty}\|f_{k-j}\|_{L^2}\right)^2\\
    &\lesssim \sum_{\substack{k=-\infty}}^{-1}\left(\sum_{\substack{j=-4}}^\infty2^{k-j}\|f\|_{L^\infty}\left\|\D_x^{-1}f\right\|_{L^2}\right)^2\\
    &\lesssim \|f\|^2_{L^\infty}\left\|\D_x^{-1}f\right\|^2_{L^2}.
\end{aligned}\end{equation*}
We conclude in the case $s > 0$ that
\begin{equation*}\begin{aligned}
   \|B_2 \|_{\dot H^s}&\lesssim K(\|f\|_{L^\infty},\|f_x\|_{L^\infty})\|\D_x^{-1}f\|_{\dot H^s}+\|f\|_{L^\infty}\left\|\D_x^{-1}f\right\|_{L^2} \\
   &\lesssim K(\|f\|_{L^\infty},\|f_x\|_{L^\infty})\|\D_x^{-1}f\|_{H^s},
\end{aligned}\end{equation*}
and
\begin{equation*}\begin{aligned}
    \|B_2 \|_{L^2} \lesssim \|f\|_{W_x^{1,\infty}}\|\D_x^{-1}f\|_{L^2}
\end{aligned}\end{equation*}
when $s=0$.

We now exchange $B_1$ for
\begin{equation}\label{preparaprod}
   \sum_{\substack{k=-\infty}}^\infty F'(f)_{<k-4}\cdot f_k.
\end{equation}
The summand of the difference
\begin{equation*}\begin{aligned}
    &\sum_{\substack{k=-\infty}}^\infty \left(F'(f)-F'(f_{\leq_h k})\right)_{<k-4}\cdot f_k
\end{aligned}\end{equation*}
can be estimated by
\begin{equation*}\begin{aligned}
  &\left\| 2^k\left(F'(f)-F'(f_{\leq_h k})\right)_{<k-4}\cdot 2^{-k}f_k\right\|_{\dot H^s} \\
  &\quad \lesssim \left\|2^k\int_0^1 \left(F''(f_{\leq_h k}+\theta((1-h)f_k+f_{>k})) \cdot ((1-h)f_k+f_{>k}) \right)_{<k-4} \,d\theta\right\|_{L^\infty}\|\D_x^{-1}f_k\|_{\dot H^s}\\
  &\quad \lesssim  \left(\int_0^1\|f_x\|_{L^\infty}\,d\theta\right)\|\D_x^{-1}f_k\|_{\dot H^s} \lesssim \|f_{x}\|_{L^\infty}\|\D_x^{-1}f_k\|_{\dot H^s}
\end{aligned}\end{equation*}
so that summing orthogonally, we have the desired bound. 

It remains to exchange \eqref{preparaprod} with $T_{F'(f)}f$. Write
\[
F'(f)_{<k-4}\cdot f_k = T_{F'(f)_{< k-4}}f_k + T_{f_k}F'(f)_{<k-4} + \Pi(F'(f)_{<k-4},f_k).
\]
The second term contributes
\begin{equation*}\begin{aligned}
    \left\|T_{f_k}F'(f)_{<k-4}\right\|_{\dot H^s} \lesssim \|\D_x f\|_{L_x^\infty}\|\D_x^{-1}f_k\|_{\dot H^s} 
\end{aligned}\end{equation*}
which sums orthogonally over $k$. The third term is estimated similarly. For the first, using cancellation with $T_{F'(f)}f_k$, it remains to estimate
\begin{equation*}\begin{aligned}
\sum_{\substack{k=-\infty}}^\infty T_{F'(f)_{\geq k-4}}f_k,
\end{aligned}\end{equation*}
which can be done by noting that
\begin{equation*}\begin{aligned}
   \left\|T_{F'(f)_{\geq k-4}} f_k\right\|_{\dot H^s} &\lesssim \|\left(F''(f)f_x\right)_{\geq k-4}\|_{L^\infty}\left\|\D_x^{-1}f_k\right\|_{\dot H^s}\lesssim \|\D_x f\|_{L^\infty}\|\D_x^{-1}f_k\|_{\dot H^s}
\end{aligned}\end{equation*}
which is estimated in a similar way as in the frequency-balanced case.
\end{proof}

We will use the following logarithmic commutator estimate:

\begin{lemma}\label{Log commutator estimate}
  We have for $s \geq 0$,
\begin{align*}
  \|\D_x [T_f,\log\vert\partial_x\vert] g \|_{H^s} &\lesssim \| f\|_{W^{1,\infty}}\|g\|_{H^s}.
\end{align*}  
\end{lemma}

\begin{proof}
We have
\begin{align*}
\mathcal F (\D_x [T_{f},\log\vert\partial_x\vert] g) (\xi)&=\int \hat{P}_{>M}(\xi)\tilde{\chi}(\xi-\eta,\xi+\eta)\hat{f}(\xi-\eta)\hat{P}_{>M}(\eta)\hat{g}(\eta)i\xi(\log|\eta|-\log|\xi|)\,d\eta\\
&=-\int \hat{P}_{>M}(\xi) \tilde{\chi}(\xi-\eta,\xi+\eta)\widehat{\D_x f}(\xi-\eta)\hat{P}_{>M}(\eta)\hat{g}(\eta)\frac{\log\left|1+\frac{\eta-\xi}{\xi}\right|}{\frac{\eta-\xi}{\xi}}\,d\eta.
\end{align*}

Observe that $\log|1 + x|/x$ is bounded away from $x = -1$. Thus, setting
\begin{align*}
    \tilde{\chi}_1(\xi-\eta,\xi+\eta)=\tilde{\chi}(\xi-\eta,\xi+\eta)\hat {P}_{>M}(\eta)\frac{\log\left|1+\frac{\eta-\xi}{\xi}\right|}{\frac{\eta-\xi}{\xi}},
\end{align*}
and observing that $(\eta - \xi)/\xi$ remains away from $-1$ on the support of $\tilde{\chi}$ and $\hat{P}_{>M}$, where $|\eta| \geq M$ and $|\xi| \lesssim |\eta|$, we see that $\tilde{\chi}_1$ is bounded. Further, since $\tilde{\chi}_1$ satisfies the same kind of bounds as $\tilde{\chi}$, we can apply paraproduct Coifman-Meyer estimates to deduce the desired result.

\end{proof}

We also recall the following variant of \cite[Lemma 2.5]{AIT}.

\begin{lemma}[Para-products]\label{l:para-prod}
Assume that $\gamma_1, \gamma_2 < 1$, $\gamma_1+\gamma_2 \geq 0$. Then
\begin{equation}\label{para-prod}
\| T_f T_g - T_{fg} \|_{\dot H^{s} \to \dot H^{s+\gamma_1+\gamma_2}} \lesssim 
\||D|^{\gamma_1}f \|_{BMO}\||D|^{\gamma_2}g\|_{BMO}.
\end{equation}
When $\gamma_2=1$,
\begin{equation}\label{para-prod-2}
\| T_f T_g - T_{fg} \|_{\dot H^{s} \to \dot H^{s+\gamma_1+1}} \lesssim 
\||D|^{\gamma_1}f \|_{BMO}\|g_x\|_{L^\infty}.
\end{equation}
\end{lemma}

\

\subsection{Difference quotients}

We denote difference quotients by 
\[
\dq^yh(x)=\frac{h(x+y)-h(x)}{y}, \qquad \sdq^yh(x)=\frac{h(x+y)-h(x)}{| y|}
\]
and the distribution $\PVy$ by
\[
\left\langle \PVy, f \right\rangle = \int_{|y|\geq 1}\frac{f(y)}{|y|}\,dy+\int_{|y|<1}\frac{f(y)-f(0)}{|y|}\,dy.
\]
Equivalently,
\begin{align*}
    \left\langle \PVy, f \right\rangle = \lim_{\eps\rightarrow 0} \left(\int_{|y|>\eps} \frac{f(y)}{|y|}\,dy+2\log(\eps)f(0)\right).
\end{align*}

We also use the following estimates which we will apply to averages over difference quotients:

\begin{lemma}\label{l:yavg}
We have
\[
\left \|\int \frac{1}{|y|} f(x, y) \, dy \right\|_{H^s_x} \lesssim \sup_{|y| > 1} \||y|^{\delta} f\|_{H^s_x} + \sup_{|y| \leq 1} \||y|^{-\delta} f\|_{H^s_x}
\]
and
\[
\left \|\int \frac{1}{|y|} f(x, y) \, dy \right\|_{L_x^\infty} \lesssim  \sup_{|y| > 1} \||y|^{\delta} f\|_{L^\infty_x} + \sup_{|y| \leq 1} \||y|^{-\delta} f\|_{L^\infty_x}.
\]
\end{lemma}

\begin{proof}

We decompose the integral
\[
\int\frac{1}{|y|} f(x, y) \, dy = \int_{|y| > 1} + \int_{|y| \leq 1}
\]
and estimate
\[
\left \|\int_{|y| > 1} \frac{1}{|y|} f(x, y) \, dy \right\|_{H^s_x} \lesssim \sup_{|y| > 1} \||y|^{\delta} f\|_{H^s_x}, \quad \left \|\int_{|y| \leq 1} \frac{1}{|y|} f(x, y) \, dy \right\|_{H^s_x} \lesssim \sup_{|y| \leq 1} \||y|^{-\delta} f\|_{H^s_x}
\]
with similar estimates for $L^\infty$.

\end{proof}

\begin{lemma}\label{Trilinear integral estimate}
Let $i = \overline{1,n}$ and $p_i, r \in [1, \infty]$ and $\alpha_i, \beta_i \in [0, 1]$ satisfying
\[
\sum_{i}\frac{1}{p_i} = \frac{1}{r}, \qquad n-1 < \sum_{i}\alpha_i\leq n, \qquad 0 \leq \sum_{i}\beta_i < n-1.
\]
Then
\[
 \left\|\int \text{sgn}(y) \prod \dq^yf_i \,dy\right\|_{L_x^r} \lesssim \prod \||D|^{\alpha_i} f_i\|_{L^{p_i}} + \prod \||D|^{\beta_i} f_i\|_{L^{p_i}}.
 \]

\end{lemma}

\begin{proof}
We write 
\[
\int \text{sgn}(y) \prod \dq^yf_i \,dy = \int_{|y| \leq 1} + \int_{|y| > 1}.
\]
For the former integral, we have by H\"older
\[
 \left\| \int_{|y| \leq 1} \text{sgn}(y) \prod \dq^yf_i \,dy \right\|_{L_x^r} \lesssim \int_{|y| \leq 1} \frac{1}{|y|^{n - \sum \alpha_i}} \prod \||D|^{\alpha_i} f_i\|_{L^{p_i}} \, dy \lesssim \prod \||D|^{\alpha_i} f_i\|_{L^{p_i}}.
\]
The latter integral is treated similarly.
\end{proof}

\section{Paralinearizations and the linearized equation}\label{s:paralinearization}

The objective of this section is to paralinearize the operator $A_\varphi v$ defined in \eqref{A-op-lin}. With this paralinearization, we reduce the analysis of \eqref{SQG} and its linearization to the analysis of a paradifferential flow with perturbative source.

The analysis in this section is time independent. We will use the notation $B^0$ to denote the zeroth order coefficient of the paralinearization,
\[
B^0(\varphi) =\langle \PVy,F(\dq^y\varphi)\rangle-F(\varphi_x)\langle\PVy,e^{-iy}\rangle.
\]

\begin{proposition}\label{paralinearization-lin}
We have the paralinearization
\[
A_\varphi v = T_{B^0(\varphi)} v - 2 T_{F(\varphi_x)} \log|D_x| v + \mathcal{R}(\varphi, v)
\]
where
\begin{align*}
    \|\D_x \mathcal{R}(\varphi, v)\|_{L^2}\lesssim_{\|\varphi_x\|_{C^{1,\delta}}}\||D_x|^{-\delta}\varphi_x\|^2_{C^{1,2\delta}}\|v\|_{L^2}.
\end{align*}
\end{proposition}

\begin{proof}

We decompose
\begin{align*}
    A_{\varphi} v &= \int T_{\sdq^yv}F(\dq^y \varphi)\,dy +\int \Pi(F(\dq^y\varphi),\sdq^yv)\,dy + \int T_{F(\dq^y\varphi)}\sdq^yv \,dy \\
    &:=A_1 + A_2 + A_3.
\end{align*}

We estimate $A_1$ using Lemma~\ref{l:yavg},
\begin{equation*}\begin{aligned}
\|\D_x A_1\|_{L^2} &\lesssim \sup_y \left( \|y^\delta \D_x T_{v(\cdot + y)- v(\cdot)}F(\dq^y \varphi)\|_{L^2} + \|y^{-\delta} \D_x T_{v(\cdot + y)- v(\cdot)}F(\dq^y \varphi)\|_{L^2}\right).
\end{aligned}\end{equation*}
For the first term, we have
\begin{equation*}
\begin{aligned}
\|y^\delta \D_x F(\dq^y\varphi)\|_{L^\infty} &\lesssim \|y^\delta \D_x \dq^y\varphi \|_{L^\infty} \|F'(\dq^y\varphi)\|_{L^\infty} \lesssim \|\D_x \varphi\|_{C^{0,1-\delta}} \|\D_x \varphi\|_{L^\infty}.
\end{aligned}
\end{equation*}
For the second term, we have 
\[
\| \D_x T_{y^{-\delta}(v(\cdot + y)- v(\cdot))}F(\dq^y \varphi)\|_{L^2} \lesssim \||D|^{-\delta} v\|_{\dot H^\delta} \||D|^{1 + \delta} F(\dq^y \varphi)\|_{BMO} \lesssim \|v\|_{L^2}\|\D_x \varphi\|_{C^{1,\delta}}^2
\]
so that the contribution from $A_1$ may be absorbed into $\mathcal R(\varphi, v)$. A similar analysis applies to absorb $A_2$ into $\mathcal R(\varphi, v)$. 

\

We proceed with $A_3$, which we may write as
\begin{align*}
   A_3 = -\int &T_{F(\dq^y\varphi)\frac{1-e^{i\eta y}}{|y|}}v \, dy=-T_{\int F(\dq^y\varphi)\frac{1-e^{i\eta y}}{| y|} \, dy}v.
\end{align*}
Since $F(\dq^y\varphi)(1-e^{i\eta y})$ vanishes at $y = 0$ and $F(\dq^y\varphi)$ satisfies the decay 
\[
|F(\dq^y\varphi)| \lesssim_{\|\varphi\|_{L^\infty}} |y|^{-2} 
\]
in $y$, we may express the symbol of the paradifferential operator in terms of $\PVy$:
\[
 \int F(\dq^y\varphi)\frac{1-e^{i\eta y}}{| y|}\,dy = \left\langle \PVy, F(\dq^y\varphi)(1-e^{-i\eta y})\right\rangle.
\]

The inverse Fourier transform of $\PVy$ takes the form
\begin{equation*}\begin{aligned}
\left\langle \PVy, e^{-i\eta y}\right\rangle &= \lim_{\eps\rightarrow 0} \left(\int_{|y|>\eps} \frac{e^{i\eta y}}{|y|}\,dy + 2\log(\eps)\right) \\
&= \lim_{\eps\rightarrow 0} \left(\int_{|z|>|\eta|\eps} \frac{e^{iz}}{|z|}\,dz + 2\log(|\eta|\eps)\right) - 2\log |\eta| \\
&= \left\langle \PVy, e^{-iy}\right\rangle - 2\log |\eta|.
\end{aligned}\end{equation*}
Using this, we write 
\begin{equation}\label{A-contribution} \begin{aligned}
 \left\langle \PVy, F(\dq^y\varphi)(1-e^{-i\eta y})\right\rangle &= \left\langle \PVy, F(\dq^y\varphi) - F(\varphi_x) e^{-i\eta y}\right\rangle \\
&\quad + \left\langle \PVy, (F(\varphi_x) - F(\dq^y\varphi)) e^{-i\eta y}\right\rangle \\
 &= B^0(\varphi) + 2F(\varphi_x) \log |\eta| \\
&\quad + \left\langle \PVy, (F(\varphi_x) - F(\dq^y\varphi)) e^{-i\eta y}\right\rangle.
\end{aligned}\end{equation}

\

Since the first two terms on the right hand side of \eqref{A-contribution} form the main terms of the paralinearization of $A_\varphi$, it remains to absorb the contribution of the last term in \eqref{A-contribution} into $\mathcal{R}$. From Lemma 6.2 in \cite{32}, we have 
\begin{equation}\label{a3analysis}
 \D_x T_{\left\langle \PVy, (F(\varphi_x) - F(\dq^y\varphi)) e^{-i\eta y}\right\rangle} v = T'_{\int \frac{F(\varphi_x) - F(\dq^y \varphi)}{|y|}i\eta e^{i\eta y}\,dy}v, 
\end{equation}
where $T'$ is another low-high paraproduct.

 Integrating the paradifferential symbol by parts,
\begin{equation}\label{A-contribution-2} \begin{aligned}
     \int &\frac{F(\varphi_x)-F(\dq^y\varphi)}{|y|}e^{i\eta y}i\eta\,dy \\
     &= \int_0^{\infty}\frac{F(\varphi_x)-F(\dq^y\varphi)}{y}\D_y(e^{i\eta y})\,dy - \int_{-\infty}^0\frac{F(\varphi_x)-F(\dq^y\varphi)}{y}\D_y(e^{i\eta y})\,dy\\
     &=  F'(\varphi_x)\varphi_{xx} - \int \sgn(y) \cdot \D_y \left(\frac{F(\varphi_x)-F(\dq^y\varphi)}{y}\right)e^{i\eta y}\,dy\\
     &=: F'(\varphi_x)\varphi_{xx} - \int G(x,y)e^{i\eta y}\,dy.
\end{aligned}\end{equation}
It is immediate to see that the contribution from the first term on the right in \eqref{A-contribution-2} may be absorbed into $\mathcal{R}$. It remains to consider the contribution from the second term, which we write
\begin{equation*}\begin{aligned}
-T'_{\int G(x,y)e^{i\eta y}\,dy} v = -\int T'_{G(x, y)} v( \cdot + y) \, dy
\end{aligned}\end{equation*}
so that
\begin{equation*}\begin{aligned}
\left\|\int T'_{G(x, y)} v( \cdot + y) \, dy \right\|_{L^2} &\lesssim \int \|G(\cdot, y)\|_{L^\infty} \, dy \cdot \|v\|_{L^2} \\
&\lesssim \left(\int_{|y| \leq 1} \|G(\cdot, y)\|_{L^\infty} \, dy + \int_{|y| > 1} \|G(\cdot, y)\|_{L^\infty} \, dy \right) \|v\|_{L^2}.
\end{aligned}\end{equation*}
Since
\begin{equation*}\begin{aligned}
G(x, y) &= \sgn (y) \frac{- y\cdot \D_y F(\dq^y\varphi) + F(\dq^y\varphi) - F(\varphi_x)}{y^2}
\end{aligned}\end{equation*}
and
\[
|\D_y F(\dq^y\varphi)| \lesssim\|\varphi_x\|^2_{L_x^\infty}|y|^{-1},
\]
we have
\[
\|G(\cdot, y)\|_{L^\infty} \lesssim \|\varphi_x\|^2_{L_x^\infty} \frac{1}{y^2}
\]
so that the integral over $\{|y| > 1\}$ converges with the appropriate bound. On the other hand, for the integral over $\{|y| \leq 1\}$, we use instead
\begin{equation*}\begin{aligned}
|y|^{1 - \delta} \|G(\cdot, y)\|_{L^\infty} &\lesssim\left\|\frac{1}{|y|^\delta} \left( \D_y F(\dq^y\varphi) - \frac{F(\dq^y\varphi) - F(\varphi_x)}{y}\right)\right\|_{L_x^\infty} \\
&\lesssim \|F(\dq^y\varphi)\|_{L_x^\infty C_y^{1,\delta}} \lesssim \|\D_x \varphi\|_{C^{1,\delta}_x}^2.
\end{aligned}\end{equation*}
which also suffices.

\end{proof}

Next, we paralinearize the nonlinear term $A_\varphi \varphi$ in \eqref{SQG}, evaluating the $H^s$ higher regularity of the errors: 

\begin{proposition}\label{paralinearization}
We have the paralinearization
\[
A_\varphi \varphi_x = \D_x T_{B^0(\varphi)}\varphi - 2 \D_x T_{F(\varphi_x)} \log|D_x| \varphi + \mathcal{R}(\varphi),
\]
where for any $s \geq 0$,
\begin{align*}
    \|\mathcal{R}(\varphi)\|_{H^s}\lesssim_{\|\varphi_x\|_{C^{1,\delta}}}\||D_x|^{-\delta}\varphi_x\|^2_{C^{1,2\delta}}\|\varphi\|_{H^s}.
\end{align*}

Moreover, if $\varphi$ and $\psi$ are two solutions and $v=\varphi-\psi$,
\begin{align*}
    \|\mathcal{R}(\varphi)-\mathcal{R}(\psi)\|_{H_x^s}&\lesssim K(\|(\varphi,\psi)\|_{W^{2 + s + \delta, \infty}})\|v\|_{L_x^2} + K(\|(\varphi,\psi)\|_{W^{2 + \delta, \infty}})\|v\|_{H_x^s} ,
\end{align*}
where $K:[0,\infty)\rightarrow[0,\infty)$ is a nondecreasing function.
\end{proposition}

We remark that the $\mathcal R$ difference estimate is not optimal with respect to the pointwise control coefficients. However, this estimate is only used in the construction of smooth solutions, and will play no role in the refined low regularity analysis.

\begin{proof}
We write
\begin{align*}
    A_{\varphi} \varphi_x &= \int T_{\sdq^y\varphi_x}F(\dq^y \varphi)\,dy +\int \Pi(F(\dq^y\varphi),\sdq^y\varphi_x)\,dy + \int T_{F(\dq^y\varphi)}\sdq^y\varphi_x \,dy \\
    &= \int T_{\sdq^y\varphi_x} F(\dq^y \varphi) - T_{\D_x (F(\dq^y \varphi))}\dq^y \varphi \,dy +\int \Pi(F(\dq^y\varphi),\sdq^y\varphi_x)\,dy \\
    &\quad + \D_x \int T_{F(\dq^y\varphi)}\sdq^y \varphi \,dy  \\
    &:=A_1 + A_2 + \D_x A_3.
\end{align*}
The analysis of $A_3$ closely follows the analysis of the counterpart $A_3$ in the proof of Proposition~\ref{paralinearization-lin}, with $\varphi$ in the place of $v$ and $H^s$ in the place of $L^2$. This contributes both terms in the paralinearization. The balanced frequency term $A_2$ is also similar to its counterpart and may be absorbed into $\mathcal R(\varphi)$. 

\

It remains to absorb $A_1$ into $\mathcal R(\varphi)$. We further decompose
\[
A_1 = \int T_{\sdq^y\varphi_x} (F(\dq^y \varphi) - T_{F'(\dq^y \varphi)}\dq^y \varphi)\,dy + \int T_{\dq^y\varphi_x} T_{F'(\dq^y \varphi)}\sdq^y \varphi - T_{\D_x (F(\dq^y \varphi))}\sdq^y \varphi \,dy .
\]
The first difference is estimated using the Moser estimates of Theorem~\ref{t:moser}, while the second may be estimated using the para-product estimates of Lemma~\ref{l:para-prod}.

\

We now prove the difference estimate.
We rewrite
\begin{align*}
    A_{\varphi} \varphi_x &= \int T_{\sdq^y\varphi_x}F(\dq^y \varphi)\,dy +\int \Pi(F(\dq^y\varphi),\sdq^y\varphi_x)\,dy + \int T_{F(\dq^y\varphi)}\sdq^y\varphi_x \,dy \\
    &:=B_1 + B_2 + B_3.
\end{align*}
We consider the components of the remainders arising from the terms of the form $B_3$. One such contribution is 
\begin{align*}
\int &T_{\frac{1}{|y|}(F(\dq^y\varphi)-F(\varphi_x))e^{i\eta y}}\varphi_x\,dy-\int T_{\frac{1}{|y|}(F(\dq^y\psi)-F(\psi_x))e^{i\eta y}}\psi_x\,dy\\
&=\int T_{\frac{1}{|y|}(F(\dq^y\varphi)-F(\varphi_x))e^{i\eta y}}v_x\,dy+\int T_{\frac{1}{|y|}((F(\dq^y\varphi)-F(\varphi_x))-(F(\dq^y\psi)-F(\psi_x)))e^{i\eta y}}\psi_x\,dy.
\end{align*}
We analyze the last term, by splitting the integral over the regions $\{|y|\leq 1\}$ and $\{|y|>1\}$, respectively. We first consider the former. We claim that
\begin{align*}
    &\left\|\int_{\vert y\vert\leq 1} \partial_xT_{(F(\dq^y\varphi)-F(\varphi_x))-(F(\dq^y\psi)-F(\psi_x))}\sdq^y\psi\,dy \right\|_{L_x^2}\lesssim \|v\|_{L_x^2}\|(\psi_x,\varphi_x)\|_{W_x^{1,\infty}}\|\psi_x\|_{W_x^{2,\infty}},
\end{align*}
where $v=\varphi-\psi$.

For this purpose, we write the low frequency component as
\begin{align*}
    D:=&\frac{1}{|y|}(F(\dq^y\varphi)-F(\varphi_x))-(F(\dq^y\psi)-F(\psi_x))\\
    &=\frac{1}{|y|}\int_0^1\varphi_x(x+\mu y)-\varphi_x(x)\, d\mu \cdot\int_0^1F'\left(\lambda\int_0^1\varphi_x(x+\mu y)\, d\mu+(1-\lambda)\varphi_x(x)\right)\, d\lambda\\
    &-\frac{1}{|y|}\int_0^1\psi_x(x+\mu y)-\psi_x(x)\, d\mu \cdot\int_0^1F'\left(\lambda\int_0^1\psi_x(x+\mu y)\, d\mu+(1-\lambda)\psi_x(x)\right)\, d\lambda.
\end{align*}
Let $\displaystyle \int_0^1 f(x+\mu y)-f(x)\, d\mu:=k_f(x,y)$, $\displaystyle p_f(x,y)=\sgn(y)\int_0^1\int_0^1 \mu f(x+\tau\mu y)-f(x)\, d\mu \, d\tau$ $\displaystyle l_f(\lambda,x,y)=\lambda k_f(x,y)+f(x)$. In particular, $\displaystyle\frac{k_f(x,y)}{|y|}=p_{f_x}(x,y)$. 
$D$ becomes
\begin{align*}
D&=\frac{1}{|y|}k_{v_x}\int_0^1F'\left(l_{\varphi_x}(\lambda,x,y)\right)\, d\lambda\\
&\quad \, +\frac{1}{|y|}k_{\psi_x} \int_0^1 l_{v_x}(\lambda,x,y)\int_0^1F''\left(\nu l_{\varphi_x}(\lambda,x,y)+(1-\nu)l_{\psi_x}(\lambda,x,y)\right)\, d\nu \, d\lambda,
\end{align*}
which can in turn be written as
\begin{align*}
D&=\partial^2_x\left(p_{v}\int_0^1F'\left(l_{\varphi_x}(\lambda,x,y)\right)\, d\lambda\right)-\partial_x\left(p_{v}\partial_x\int_0^1F'\left(l_{\varphi_x}(\lambda,x,y)\right)\, d\lambda\right)\\
    & \quad -\partial_x\left(p_{v}\partial_x\int_0^1F'\left(l_{\varphi_x}(\lambda,x,y)\right)\, d\lambda\right)+p_{v}\partial^2_x\int_0^1F'\left(l_{\varphi_x}(\lambda,x,y)\right)\, d\lambda\\
    & \quad +\partial_x\left(p_{\psi_{xx}} \int_0^1 l_{v}(\lambda,x,y)\int_0^1F''\left(\nu l_{\varphi_x}(\lambda,x,y)+(1-\nu)l_{\psi_x}(\lambda,x,y)\right)\, d\nu \, d\lambda\right)\\
    &\quad  - \int_0^1 l_{v}(\lambda,x,y)\partial_x\left(p_{\psi_{xx}}\int_0^1F''\left(\nu l_{\varphi_x}(\lambda,x,y)+(1-\nu)l_{\psi_x}(\lambda,x,y)\right)\right)\, d\nu \, d\lambda .
\end{align*}
We can now move the derivatives in the first three terms from the low frequencies to the high frequencies and obtain the claimed estimates.

For the $\{|y|>1\}$ region, we write
\begin{align*}
T_{\int_{|y|>1}\frac{1}{|y|}((F(\dq^y\varphi)-F(\varphi_x))-(F(\dq^y\psi)-F(\psi_x))e^{i\eta y}i\eta)}\psi
\end{align*}
integrate the low frequency component by parts (and use the notation of \eqref{A-contribution-2})

\begin{align*}
\int_{|y|>1}\frac{1}{|y|}&((F(\dq^y\varphi)-F(\varphi_x))-(F(\dq^y\psi)-F(\psi_x))\partial_ye^{i\eta y})\\
&=(2F(\varphi_x)-F(\dq^1\varphi)-F(\dq^{-1}\varphi))-(2F(\psi_x)-F(\dq^1\psi)-F(\dq^{-1}\psi))\\
&\quad -T_{\int_{|y|>1} (G_\varphi(x, y) - G_\psi(x, y) )e^{i\eta y}\,dy}\psi,
\end{align*}
where $G_\varphi$ is defined in the proof of Proposition \ref{paralinearization-lin}.

The estimate for this term now follows using ideas similar to the ones for the region $\{|y|\leq 1\}$ and from the proof of \ref{paralinearization-lin}.

From the contribution corresponding to $A_1$, we also have
\begin{align*}
    &\int_{\vert y\vert\leq 1}T_{\sdq^y\varphi_x}F(\dq^y\varphi)-T_{F'(\dq^y\varphi)\dq^y\varphi_x}\sdq^y\varphi-T_{\sdq^y\psi_x}F(\dq^y\psi)+T_{F'(\dq^y\psi)\dq^y\psi_x}\sdq^y\psi\,dy\\
    &=\int_{\vert y\vert\leq 1}T_{\sdq^y\varphi_x-\sdq^y\psi_x}F(\dq^y\varphi)\,dy-\int_{\vert y\vert\leq 1}T_{F'(\dq^y\varphi)\dq^y\varphi_x-F'(\dq^y\psi)\dq^y\psi_x}\sdq^y\varphi\,dy\\
    &\quad +\int_{\vert y\vert\leq 1}T_{\sdq^y\psi_x}F(\dq^y\varphi)-F(\dq^y\psi)\,dy
    -\int_{\vert y\vert\leq 1}T_{F'(\dq^y\psi)\dq^y\psi_x}\sdq^y(\varphi-\psi)\,dy.
\end{align*}
 As $F$ has a zero of order $2$ at $0$, Lemma \ref{Trilinear integral estimate} implies that
\begin{align*}
    \left\|\int_{\vert y\vert\leq 1}T_{\sdq^yv_x}F(\dq^y\varphi)\,dy\right\|_{L_x^2}&\lesssim\|v\|_{L_x^2}\|\varphi_x\|^2_{L_x^\infty}((\|\varphi_x\|_{L_x^\infty}+1) (\|\vert D_x\vert^{1+\delta}\varphi_x\|_{L_x^\infty}+\|\varphi_{xx}\|_{L_x^\infty})\\
   &\quad +\|\vert D_x\vert^{1+\delta}\varphi_x\|_{L_x^\infty}+\|\vert D_x\vert^{1-\delta}\varphi_x\|_{L_x^\infty}+\|\varphi_x\|_{L_x^\infty}).
\end{align*}
The other bounds follow by reasoning similarly as before.

The other terms in $\mathcal{R}(\varphi)-\mathcal{R}(\psi)$ can be treated similarly. The proof of the $H_x^s$-bound is analogous.

\end{proof}

Recall that the linearized equation \eqref{SQG-lin} corresponding to \eqref{SQG} may be written
\begin{equation*}
    \partial_tv - \D_x A_\varphi v =2\log|D_x| \D_xv.
\end{equation*}
Using the paralinearization Proposition~\ref{paralinearization-lin}, we obtain
\begin{proposition}\label{Linearized equation estimate}
The linearized equation \eqref{SQG-lin} admits the paralinearization
\begin{equation}\label{SQG-paralin}
    \partial_tv - \D_xT_{B^0(\varphi)}v +\mathcal{R}(\varphi,v) = 2\D_x T_{1 - F(\varphi_x)} \log|D_x|v,
\end{equation}
where
\begin{align*}
    \|\mathcal{R}(\varphi, v)\|_{L^2}\lesssim_{\|\varphi_x\|_{C^{1,\delta}}}\||D_x|^{-\delta}\varphi_x\|^2_{C^{1,2\delta}}\|v\|_{L^2}.
\end{align*}

\end{proposition}

\section{Energy estimates}\label{s:energy}

The goal of this section is to prove energy estimates for the paradifferential flow \eqref{Linearized equation estimate}, which will in turn be used to derive energy estimates for solutions of \eqref{SQG}, along with its linearization \eqref{SQG-lin}.

We define the modified energies
\[
E^{(s)}(v) = \int g\cdot T_{1 - F(\varphi_x)} g \,dx, \qquad g = T_{(1 - F(\varphi_x))^s} | D_x|^s v,
\]
and
\begin{align*}
E^{s}(v) =E^{(s)}(v)+E^{(0)}(v).
\end{align*}
Due to our construction of the paraproduct, these modified energies are comparable to the classical energies in the Sobolev spaces $H^s(\mathbb{R})$. 
 
We prove the following:
\begin{proposition}\label{p:linear-exist}
Let $v_0\in H^s$ and $\mathcal R \in C([0,T],H_x^s)$, $\varphi \in C([0, T], W_x^{2,\infty})$. Then the initial value problem
\begin{equation}\begin{aligned}\label{lin-eqn}
    \D_tv - \D_xT_{B^0(\varphi)}v +\mathcal{R} &= 2\D_x T_{1 - F(\varphi_x)} \log|D_x|v,\\
v(0,x)&=v_0(x)
\end{aligned}\end{equation}
has a unique solution $v \in C([0,T];H_x^s)$, satisfying the estimate
\begin{equation}\begin{aligned}\label{lin-eqn-energy}
  \frac{d}{dt}E^{s}(t)&\lesssim (\|\varphi_x\|_{C^{1, \delta}}+\|\varphi_{tx}\|_{L^\infty})\|\varphi_x\|_{C^{1, \delta}} E^{s}(t) + \|\mathcal{R}\|_{H^s}\|v\|_{H^s}.
\end{aligned}\end{equation}
\end{proposition}

\begin{proof}
Using that $T_{1 - F(\varphi_x)}$ is self-adjoint,
\begin{align*}
    \frac{d}{dt}E^{(s)}(t)&=\int 2 g_t\cdot T_{1 - F(\varphi_x)} g - g\cdot T_{F'(\varphi_x)\varphi_{tx}} g\,dx.
\end{align*}

Applying $T_{(1 - F(\varphi_x))^s} |D_x|^s$, we obtain an equation for $g_t$,
\[
  \D_tg - \D_xT_{B^0(\varphi)}g + T_{(1 - F(\varphi_x))^s} |D_x|^s \mathcal{R} + \mathcal R_1 = 2\D_x \log|D_x|T_{1 - F(\varphi_x)} g,
\]
where $\mathcal R_1$ consists of commutators which we will record and estimate momentarily. Using this, we have
\begin{equation}\label{dt-energy}
\begin{aligned}
    \frac{d}{dt}E^{(s)}(t)&=\int 2 (\D_xT_{B^0(\varphi)}g + 2\D_x \log|D_x|T_{1 - F(\varphi_x)} g -\mathcal{R} - \mathcal R_1)\cdot T_{1 - F(\varphi_x)} g \, dx \\
    &\quad - \int g\cdot T_{F'(\varphi_x)\varphi_{tx}} g\,dx.
\end{aligned}
\end{equation}
From the first integral on the right hand side, we rewrite the contribution
\begin{equation*}\begin{aligned}
2&\int \D_xT_{B^0(\varphi)}g\cdot T_{1 - F(\varphi_x)} g\,dx = 2\int \D_x(T_{B^0(\varphi)}g\cdot T_{1 - F(\varphi_x)} g)\,dx + \int \mathcal R_2 \cdot g\,dx,
\end{aligned}\end{equation*}
where as before, $\mathcal R_2$ consists of commutators which we record and estimate momentarily. Observe that the first term on the right is a divergence which thus vanishes. Similarly, since $\D_x \log|D_x|$ is skew-adjoint, the corresponding contribution to \eqref{dt-energy} vanishes and we conclude 
\begin{equation*}\begin{aligned}
    \frac{d}{dt}E^{(s)}(t)&= - \int 2 (\mathcal{R} + \mathcal R_1)\cdot T_{1 - F(\varphi_x)} g + g\cdot T_{F'(\varphi_x)\varphi_{tx}} g\,dx + \int \mathcal R_2 \cdot g\,dx ,
\end{aligned}\end{equation*}
where
\begin{equation*}\begin{aligned}
\mathcal R_1 &= [T_{(1 - F(\varphi_x))^s} |D_x|^s, \D_t - \D_xT_{B^0(\varphi)}] v \\
&\quad + 2\D_x [T_{F(\varphi_x)},\log|D_x|]g - 2[T_{(1 - F(\varphi_x))^s} |D_x|^s , \D_x \log|D_x|T_{1 - F(\varphi_x)}] v, \\
\mathcal R_2 &= \left(T_{1 - F(\varphi_x)} T_{\D_x B^0(\varphi)} + [T_{1 - F(\varphi_x)}, T_{B^0(\varphi)}] \D_x + T_{\D_x  F(\varphi_x)} T_{B^0(\varphi)}  \right)g.
\end{aligned}\end{equation*}

\

The second commutator of $\mathcal R_1$ may be estimated using the $\log |D_x|$ commutator Lemma~\ref{Log commutator estimate}. For the third commutator of $\mathcal R_1$, applying again Lemma~\ref{Log commutator estimate}, we may commute $\log|D_x|$ to the front, so this reduces to 
\begin{equation*}\begin{aligned}
\log|D_x| [T_{(1 - F(\varphi_x))^s} | D_x|^s, \D_x T_{1 - F(\varphi_x)}] v &= \log|D_x| [T_{(1 - F(\varphi_x))^s}, \D_x] | D_x|^s T_{1 - F(\varphi_x)} v \\
&\quad + \log|D_x|\D_x T_{(1 - F(\varphi_x))^s} [| D_x|^s, T_{1 - F(\varphi_x)}] v \\
&\quad + \log|D_x|\D_x [T_{(1 - F(\varphi_x))^s}, T_{1 - F(\varphi_x)}]|D_x|^s v.
\end{aligned}\end{equation*}
The third commutator on the right may be estimated using Lemma~\ref{l:para-prod}. The first may be written
\[
\log|D_x| T_{s(1 - F(\varphi_x))^{s - 1} F'(\varphi_x) \varphi_{xx}}|D_x|^s T_{1 - F(\varphi_x)} v,
\]
while the principal term of the second is
\[
-\log|D_x|\D_x T_{(1 - F(\varphi_x))^s} s|D_x|^s \D_x^{-1} T_{F'(\varphi_x) \varphi_{xx}} v,
\]
which cancel up to commutators and paraproducts estimated via Lemma~\ref{l:para-prod}.

Returning to $\mathcal R_1$, it remains to consider the first commutator. For the commutator with respect to $\D_t$, we have
\[
T_{s(1 - F(\varphi_x))^{s - 1} F'(\varphi_x) \varphi_{tx}} |D_x|^s  v
\]
which contributes to the right hand side of \eqref{lin-eqn-energy}. For the commutator with respect to $\D_xT_{B^0(\varphi)}$, and in particular to estimate $B^0(\varphi)$, we apply Lemma~\ref{l:yavg} to estimate
\begin{equation*}\begin{aligned}
\|B^0(\varphi)\|_{L_x^\infty} &\lesssim  \sup_{|y| > 1} \|y^{\delta} (F(\dq^y\varphi) - F(\varphi_x) e^{-iy})\|_{L^\infty_x} + \sup_{|y| \leq 1} \|y^{-\delta} (F(\dq^y\varphi) - F(\varphi_x) e^{-iy})\|_{L^\infty_x} \\
&\lesssim \|\varphi_x\|_{C^{0, \delta}},
\end{aligned}\end{equation*}
using the decay of $F$ for the first term. The same estimate for $B^0(\varphi)$ then suffices to estimate each term of $\mathcal R_2$. 

\

We conclude, using Lemma~\ref{M-paraprod} to pass between $\|v\|_{H^s}$ and $\|g\|_{L^2}$,
\begin{align*}
  \frac{d}{dt}E^{(s)}(t)&\lesssim (\|\varphi_x\|_{C^{1}}+\|\varphi_{tx}\|_{L^\infty})\|\varphi_x\|_{C^{1, \delta}}\|g\|^2_{L^2} + \| \mathcal{R}\|_{H^s}\|g\|_{L^2}
\end{align*}
which along with the similar case $s=0$ gives the estimate of the proposition. By using the adjoint method, it follows that the equation has a unique solution.
\end{proof}
\begin{corollary}\label{Energy Estimates}
   Let $R>0$. If $\varphi$ is a solution of \eqref{SQG} on an interval $[0,T]$ on which $\|\varphi\|_{C_t^0H_x^s}\lesssim R$, where $s>\frac{5}{2}$, then we have the energy estimate
   \begin{align*}
\|\varphi(t,x)\|_{H_x^s}\leq C(R)e^{C(R)\int_0^t\||D_x|^{-\delta}\varphi_x(\tau)\|^2_{C^{1,2\delta}}\,d\tau}\|\varphi_0\|_{H_x^s}.
   \end{align*}
   Moreover, if $v$ is a solution of the linearized equation \eqref{SQG-lin} on an interval $[0,T]$ where the solution $\varphi$ satisfies the previous conditions, then
   \begin{align*}
\|v(t,x)\|_{L_x^2}\leq C(R)e^{C(R)\int_0^t\||D_x|^{-\delta}\varphi_x(\tau)\|^2_{C^{1,2\delta}}\,d\tau}\|v_0\|_{L_x^2}.
   \end{align*}
\end{corollary}
\begin{proof}
From Proposition \ref{p:linear-exist}, we have the energy estimate
\begin{align*}
\frac{d}{dt}E^{(s)}(t)&\lesssim (\|\varphi_x\|_{C^{1, \delta}}+\|\varphi_{tx}\|_{L^\infty})\|\varphi_x\|_{C^{1, \delta}} E^{(s)}(t) + \|\mathcal{R}\|_{H^s}\|v\|_{H^s}.
\end{align*}
In both situations, $\varphi$ is a solution of \eqref{SQG}, so in order to control $\displaystyle \|\varphi_{tx}\|_{L_x^\infty}$, we write
\begin{align*}
    \varphi_{tx}=2\log|D_x|\varphi_{xx}-\partial_x\int F(\dq^y\varphi)\sdq^y\varphi_x\,dy.
\end{align*}
From Lemma \ref{Trilinear integral estimate} and the product rule, we have
\begin{align*}
    \left\|\partial_x\int F(\dq^y\varphi)\sdq^y\varphi_x\,dy\right\|_{L_x^\infty}&\lesssim\|\varphi_x\|_{L_x^\infty}\|\varphi_x\|_{W_x^{1,\infty}}\||D_x|^{1-\delta}\varphi\|_{W_x^{1+2\delta,\infty}}.
\end{align*}
We also have
\begin{align*}
    \|\log|D_x|\varphi_{xx}\|_{L_x^\infty}&\lesssim \||D_x|^\delta\varphi_{xx}\|_{L_x^\infty}+\||D_x|^{1-\delta}\varphi_{x}\|_{L_x^\infty}.
\end{align*}
In the first case, $v=\varphi$ solves the \eqref{SQG} equation, so by Proposition \ref{paralinearization}, we know that 
\begin{align*}
    \|\mathcal{R}(\varphi)\|_{H^s}\lesssim_{\|\varphi_x\|_{C^{1,\delta}}}\|\varphi_x\|^2_{C^{1,\delta}}\|\varphi\|_{H^s},
\end{align*}
and an application of Gr\"onwall's lemma, along with the coercivity bounds implies the claimed estimate. When $v$ solves the linearized equation \eqref{SQG-lin}, Proposition \ref{paralinearization-lin} similarly implies that
\begin{align*}
\|\mathcal{R}(\varphi, v)\|_{L^2}\lesssim_{\|\varphi_x\|_{C^{1,\delta}}}\|\varphi_x\|^2_{C^{1,\delta}}\|v\|_{L^2}.
\end{align*}

\end{proof}\section{Local well-posedness}\label{s:lwp} 

In this section we establish Theorem~\ref{t:lwp}, our main local well-posedness result in Sobolev spaces for the SQG equation \eqref{SQG}. We do this by first constructing smooth solutions using an iterative scheme, and then we employ frequency envelopes in order to construct rough solutions as limits of smooth ones and to prove continuous dependence on the initial data.

\

We first consider smoother data $\varphi_0 \in H^s$ with $s\geq 4$. Fix $R>0$ and choose $M$ as in Lemma~\ref{M-paraprod}. We construct the sequence 
\[
\varphi^{(0)}=\varphi_0(x), \qquad \varphi^{(n)}=\mathbf{G}(\varphi^{(n-1)}), \qquad n\geq 1,
\]
where we define the operator $\mathbf{G}(\varphi) = v$ using Proposition~\ref{p:linear-exist} with $\mathcal{R}=\mathcal{R}(\varphi)$, the paralinearization error from Proposition~\ref{paralinearization}. We have
\[
\D_t\varphi^{(n)} - \D_xT_{B^0(\varphi^{(n-1)})}\varphi^{(n)} +\mathcal{R}(\varphi^{(n-1)}) = 2\D_x T_{1 - F(\varphi^{(n-1)}_x)} \log|D_x|\varphi^{(n)}.
\]
An application of Gr\"onwall's inequality with Proposition~\ref{p:linear-exist} shows that
\begin{align*}
    E^{(s)}(\varphi^{(n + 1)})(t) &\leq e^{C(R)\int_0^t (\|\varphi^{(n)}_x\|_{C_x^{1,\delta}}+\|\varphi^{(n)}_{tx}\|_{L_x^\infty})\|\varphi^{(n)}_x\|_{C_x^{1,\delta}} \,d\tau} \\
    &\quad \cdot \left(E^{(s)}(\varphi^{(n + 1)})(0) + \int_0^t \|\varphi^{(n)}_x\|^2_{C_x^{1,\delta}}\|\varphi^{(n)}\|^2_{H_x^{s}}\,d\tau\right),
\end{align*}
along with an easy induction on $n$ and the terms of the sequence $(\|\varphi^{(n)}_{tx}\|_{L_x^\infty})_{n\geq 0}$ show that for sufficiently small $T>0$ depending on $R$, $\varphi^{(n)}$ has uniform $H_x^s$ bounds.

Moreover, we show that for sufficiently small $T>0$, $\varphi^{(n)}$ is Cauchy. Let $\varphi^{(m)}$ and $\varphi^{(n)}$ be two terms of the sequence. Denoting $v = \varphi^{(m)} - \varphi^{(n)}$, we have
\begin{align*}
    \partial_t v& - \partial_xT_{B^0(\varphi^{(m-1)})}v +\partial_xT_{B^0(\varphi^{(m-1)})-B^0(\varphi^{(n-1)})}\varphi^{(n)} + \mathcal{R}(\varphi^{(m-1)})-\mathcal{R}(\varphi^{(n-1)})\\
    &=2\D_x T_{1 - F(\varphi^{(m-1)}_x)} \log|D_x|v + 2\D_x T_{F(\varphi^{(m-1)}_x) - F(\varphi^{(n-1)}_x)} \log|D_x|\varphi^{(n)}.
\end{align*}
We now apply the energy estimate of Proposition~\ref{p:linear-exist} with 
\begin{align*}
    \mathcal R &=\partial_xT_{B^0(\varphi^{(m-1)})-B^0(\varphi^{(n-1)})}\varphi^{(n)} - 2\D_x T_{F(\varphi^{(m-1)}_x) - F(\varphi^{(n-1)}_x)} \log|D_x|\varphi^{(n)} \\
    &\quad + \mathcal{R}(\varphi^{(m-1)})-\mathcal{R}(\varphi^{(n-1)}),
\end{align*}
estimating
\begin{align*}
    &\|\D_x T_{F(\varphi^{(m-1)}_x) - F(\varphi^{(n-1)}_x)} \log|D_x|\varphi^{(n)}\|_{L_x^2}\lesssim \|\varphi^{(m-1)} - \varphi^{(n-1)}\|_{L_x^2}\|\log|D_x|\varphi^{(n)}_{xx}\|_{L_x^\infty}, \\
    &\|\partial_xT_{B^0(\varphi^{(m-1)})-B^0(\varphi^{(n-1)})}\varphi^{(n)} \|_{L_x^2}\lesssim \|\varphi^{(m-1)} - \varphi^{(n-1)}\|_{L_x^{2}}\|\varphi^{(n)}_{xx}\|_{L_x^\infty}.
\end{align*}
An application of Gr\"onwall's inequality shows that the sequence is Cauchy in $ L_x^2$, for $T>0$ small enough. This settles the existence. Uniqueness follows from the energy estimate for the difference.

\

To establish the local well-posedness result at low regularity, we follow the approach outlined in \cite{ITprimer}. We consider $\varphi_0 \in H^s$ with $s>\frac52$. Let $\varphi_0^h=(\varphi_0)_{\leq h}$, where $h\in\mathbb{Z}$. Since $\varphi_0^h \rightarrow u_0$ in $H_x^s$, we may assume that $ \|\varphi_0^h\|_{H_x^s}<R$ for all $h$.

We construct a uniform $H_x^s $ frequency envelope $\{c_k\}_{k\in\mathbb{Z}}$ for $\varphi_0$ having the following properties:

\begin{enumerate}
     \item[a)]Uniform bounds:     
     \[ \|P_k(\varphi_0^h)_x\|_{H_x^s}\lesssim c_k,\]
     
     \item[b)]High frequency bounds:     
     \[\|\varphi_0^h\|_{H_x^N}\lesssim 2^{h(N-s)}c_h, \qquad N - s \geq 4, \]
     
     \item[c)]Difference bounds:     
     \[\|\varphi_0^{h+1}-\varphi_0^h\|_{\dot H^s}\lesssim 2^{-sh}c_h,\]
     
     \item[d)]Limit as $h\rightarrow\infty$:     
     \[ \varphi_0^h\rightarrow \varphi_0 \in H_x^s.\]
     
 \end{enumerate}

Let $\varphi^h$ be the solutions with initial data $\varphi_0^h$. Using the energy estimate for the solution $\varphi$ of \eqref{SQG} from Corollary \ref{Energy Estimates}, we deduce that there exists $T = T(\|\varphi_0\|_{H_x^s}) > 0$ on which all of these solutions are defined, with high frequency bounds
    \[ 
    \|\varphi^h\|_{C_t^0H_x^N}\lesssim \|\varphi_0^h\|_{H_x^N} \lesssim 2^{h(N-s)}c_h.
     \]
Further, by using the energy estimates for the solution of the linearized from Corollary \ref{Energy Estimates}, we have
\[
\|\varphi^{h+1}-\varphi^h\|_{C_t^0L_x^2}\lesssim 2^{-sh}c_h.
\]
By interpolation, we infer that
\[
\|\varphi^{h+1}-\varphi^h\|_{C_t^0H_x^{s}}\lesssim c_h.
\]

As in \cite{ITprimer}, we get
\[
\|P_k \varphi^h \|_{C_t^0H_x^s } \lesssim c_k
\]
and that
\[
\|\varphi^{h+k}-\varphi^h\|_{C_t^0H_x^s }\lesssim c_{h\leq\cdot<h+k}=\left(\sum_{\substack{n=h}}^{h+k-1}c_n^2\right)^{\frac{1}{2}}
\]
for every $k\geq 1$. Thus, $\varphi^h$ converges to an element $\varphi$ belonging to $C_t^0H_x^s([0,T]\times\mathbb{R})$.  Moreover, we also obtain
\begin{equation}\label{convergence estimate}
\begin{aligned}
\|\varphi^h - \varphi\|_{C_t^0H_x^s} &\lesssim c_{\geq h}=\left(\sum_{\substack{n=h}}^{\infty} c_n^2\right)^{\frac{1}{2}}.
\end{aligned}
\end{equation}

We now prove continuity with respect to the initial data. We consider a sequence
 \[
 \varphi_{0j}\rightarrow \varphi_0 \in H_x^s
 \]
 and an associated sequence of $ H_x^s$-frequency envelopes $\{c^j_k\}_{k \in \Z}$, each satisfying the analogous properties enumerated above for $c_k$, and further such that $c^j_k \rightarrow c_k$ in $l^2(\mathbb{Z})$. In particular,
\begin{equation}\label{convergence estimate j}
\|\varphi_j^h - \varphi_j\|_{C_t^0H_x^s} \lesssim c^j_{\geq h}=\left(\sum_{\substack{n=h}}^{\infty} (c^j_n)^2\right)^{\frac{1}{2}}.
\end{equation}
 
Using the triangle inequality with \eqref{convergence estimate} and \eqref{convergence estimate j}, we write
\begin{align*}
\|\varphi_j - \varphi\|_{C_t^0H_x^s } &\lesssim \|\varphi^h - \varphi\|_{C_t^0H_x^s}+\|\varphi_j^h - \varphi_j\|_{C_t^0H_x^s}+\|\varphi_j^h - \varphi^h\|_{C_t^0H_x^s}\\
&\lesssim c_{\geq h}+c^j_{\geq h}+\|\varphi_j^h - \varphi^h\|_{C_t^0H_x^s}.
\end{align*}
To address the third term, we observe that for every fixed $h$, $\varphi_j^h \rightarrow \varphi^h$ in $H_x^s$. We conclude $\varphi_j \rightarrow \varphi$ in $C_t^0H_x^s ([0,T]\times\mathbb{R})$ and therefore $\varphi_j\rightarrow \varphi$ in $C_t^0 X^s([0,T]\times\mathbb{R})$.

\section{Global well-posedness}\label{s:gwp}

In this section we prove global well-posedness for the SQG equation \eqref{SQG} with small and localized initial data. We do this by using the wave packet method of Ifrim-Tataru, which is systematically described in \cite{ITpax}.

\subsection{Notation}

Consider the linear flow
\[
i\D_t \varphi - A(D)\varphi = 0
\]
and the linear operator
\[
L = x - tA'(D).
\]
In our setting, we have the symbol
\[
a(\xi) = -2\xi \log |\xi|
\]
and thus
\[
A(D)=-2D\log|D|, \qquad  L = x+2t+2t\log|D|.
\]
Recall that we define the weighted energy space
\[
\|\varphi\|_X = \|\varphi\|_{H^s} + \|L\partial_x \varphi\|_{L^2} \approx \|\varphi\|_{H^s} + \|\partial_x L \varphi\|_{L^2},
\]
which when frequency localized may be written
\[
\|\varphi_\lambda\|_X \approx \|\varphi_\lambda\|_{H^s} + \lambda\|L\varphi_\lambda\|_{L^2}.
\]

\

We partition the frequency space into dyadic intervals $I_\lambda$ localized at dyadic frequencies $\lambda \in 2^\Z$, and consider the associated partition of velocities
\[
J_\lambda = a'(I_\lambda)
\]
which form a covering of the real line, and have equal lengths. To these intervals $J_\lambda$ we select reference points $v_\lambda \in J_\lambda$, and consider an associated spatial partition of unity
\[
1 = \sum_\lambda \chi_\lambda(x), \qquad \text{supp } \chi_\lambda \subseteq  \overline{J_\lambda},\qquad \chi_\lambda=1\text{ on }J_\lambda,
\]
where $\overline{J_\lambda}$ is a slight enlargement of $J_\lambda$, of comparable length, uniformly in $\lambda$.

Lastly, we consider the related spatial intervals, $tJ_\lambda$, with reference points $x_\lambda  = tv_\lambda \in tJ_\lambda$. 
\

\subsection{Overview of the proof}

We provide a brief overview of the proof.

\

1. We make the bootstrap assumption for the pointwise bound
\begin{equation}\label{bootstrap}
    \|\varphi(t)\|_Y \lesssim C \eps \langle t\rangle^{-\frac{1}{2}}
\end{equation}
where $C$ is a large constant, in a time interval $t \in [0, T]$ where $T>1$. 

\

2. The energy estimates for \eqref{SQG} and the linearized equation will imply
\begin{equation}\label{Regular Energy Estimate}
    \|\varphi(t)\|_{X} \lesssim \langle t\rangle^{C^2\eps^2} \|\varphi(0)\|_X.
    \end{equation}
    
\

3. We aim to improve the bootstrap estimate \eqref{bootstrap} to 
\begin{equation}\label{pointwise}
    \|\varphi(t)\|_Y \lesssim \eps \langle t\rangle^{-\frac12}.
\end{equation}
We use vector field inequalities to derive bounds of the form
\begin{equation}\label{pt-estimate}
    \|\varphi(t)\|_Y\lesssim \eps \langle t\rangle^{-\frac{1}{2} + C\eps^2},
\end{equation}
which is the desired bound but with an extra $t^{C\eps^2}$ loss.

\

4. In order to rectify the extra loss, we use the wave packet testing method define a suitable asymptotic profile $\gamma$, which is then shown to be an approximate solution for an ordinary differential equation. This enables us to obtain suitable bounds for the asymptotic profile without the aforementioned loss, which can then be transferred back to the solution $\varphi$.

\subsection{Energy estimates} 

From Corollary \ref{Energy Estimates}, and by using the fact that $\eps\ll 1$,
\begin{align*}
 \|\varphi(t,x)\|_{H_x^s}\lesssim e^{C\int_0^t\|\varphi_x\|^2_{C_x^{1,\delta}}\,d\tau}\|\varphi_0\|_{H_x^s}. 
\end{align*}

Let $u=L\partial_x\varphi+\varphi$, which satisfies the linearized equation. From Corollary \ref{Energy Estimates}, along with Gr\"onwall's lemma and the fact that $\eps\ll 1$, we have
\begin{align*}
 \|u(t,x)\|_{L_x^2}\lesssim e^{C\int_0^t\|\varphi_x(\tau)\|^2_{C_x^{1,\delta}}\,d\tau}\|u_0\|_{L_x^2}. 
\end{align*}
Along with the bootstrap assumptions, these readily imply that
\begin{equation}\label{Vector Field Energy Estimate}
    \|\varphi\|_{X}\lesssim \|\varphi(t)\|_{H_x^s} + \|u(t)\|_{L_x^2}\lesssim \eps e^{C^2\eps^2\int_0^t\langle s\rangle^{-1}\,ds}\lesssim \eps \langle t\rangle^{C^2\eps^2}.
\end{equation}

\

\subsection{Vector field bounds} 
Proposition 2.1 from \cite{ITpax} implies that
\begin{align*}
    \|\varphi_\lambda\|^2_{L_x^{\infty}}&\lesssim \frac{1}{t}(\|\varphi_\lambda\|_{L_x^2}\|L\partial_x\varphi_\lambda\|_{L_x^2}+\|\varphi_\lambda\|^2_{L_x^2}).
\end{align*}
When $\lambda\leq 1$,
\begin{align*}
  \|\varphi_\lambda\|_{L_x^{\infty}}
  &\lesssim \frac{1}{\sqrt{t}}\lambda^{-\frac34 + 2\delta}(\|\lambda^{\frac32 - 4\delta}\varphi_\lambda\|^{1/2}_{L_x^2}\|L\partial_x\varphi_\lambda\|^{1/2}_{L_x^2}+\|\lambda^{\frac34-2\delta}\varphi_\lambda\|_{L_x^2}) \lesssim \frac{1}{\sqrt{t}}\lambda^{-(\frac34 -2\delta)}\|\varphi\|_X
\end{align*}
and when $\lambda>1$,
\begin{align*}
  \|\varphi_\lambda\|_{L_x^{\infty}}
  &\lesssim \frac{1}{\sqrt{t}}\lambda^{-2 - 2 \delta}(\|\lambda^{4+4\delta}\varphi_l\|^{1/2}_{L_x^2}\|L\partial_x\varphi_\lambda\|^{1/2}_{L_x^2}+\|\lambda^{2+2\delta}\varphi_\lambda\|_{L_x^2}) \lesssim \frac{1}{\sqrt{t}}\lambda^{-(2+2\delta)}\|\varphi\|_X.
\end{align*}
By dyadic summation and Bernstein's inequality, we deduce the bound
\begin{equation}\label{Pointwise Vector Field Bound 2}
    \|\varphi\|_Y=\||D_x|^{3/4-\delta}\varphi\|_{L_x^{\infty}}+\||D_x|^{1+\delta}\varphi_x\|_{L_x^{\infty}}\lesssim\frac{\|\varphi\|_X}{\sqrt{t}}.
\end{equation}

By the localized dispersive estimate \cite[Proposition 5.1]{ITpax}, 
\begin{align*}
    |\varphi_\lambda(x)|^2\lesssim\frac{1}{|x-x_{\lambda}|t\frac{1}{\lambda}}(\|L\varphi_\lambda\|_{L_x^2}+\lambda^{-1}\|\varphi_\lambda\|_{L_x^2})^2,
\end{align*}
 which implies that
 \begin{equation}\label{Pointwise Elliptic Estimate}
    \|(1-\chi_{\lambda})\varphi_\lambda\|_{L_x^{\infty}}\lesssim \frac{\lambda^{-1/2}}{t}(\|L\partial_x\varphi_\lambda\|_{L_x^2}+\|\varphi_\lambda\|_{L_x^2})\lesssim\frac{\lambda^{-1/2}}{t}\|\varphi\|_X
\end{equation}
 
\

To end this section we record the following elliptic bounds:

\begin{lemma}\label{Elliptic bounds for the derivative}
We have
\begin{equation}
\label{Pointwise Elliptic Estimate 3}
    \||D_x|^{1+\delta}\partial_x((1-\chi_{\lambda})\varphi_\lambda)\|_{L_x^{\infty}}\lesssim \frac{\lambda^{3/2+\delta}}{t}\|\varphi\|_X
    \end{equation}
    \begin{equation}
\label{Pointwise Elliptic Estimate 2}
   \||D_x|^{3/4-\delta}((1-\chi_{\lambda})\varphi_\lambda)\|_{L_x^{\infty}}\lesssim \frac{\lambda^{1/4-\delta}}{t}\|\varphi\|_X,
\end{equation}
and
\begin{equation}\label{L2 Elliptic Estimate}
\|(1-\chi_{\lambda})\varphi_\lambda\|_{L_x^{2}}\lesssim \frac{\lambda^{-1}}{t}\|\varphi\|_X,
\end{equation}
Moreover, the difference quotient satisfies the bounds
  \begin{align*}
\|(1-\chi_{\lambda})\dq^y\varphi_\lambda\|_{L_x^{\infty}}&\lesssim  \frac{\lambda^{1/2}}{t}\|\varphi\|_X,
  \end{align*}
  and
  \begin{align*}
\|(1-\chi_{\lambda})\dq^y\varphi_\lambda\|_{L_x^{2}}&\lesssim  \frac{\|\varphi\|_X}{t}.
  \end{align*}
\end{lemma}
\begin{proof}
We use the bounds
\begin{align*}
    |\partial_x(\chi_{\lambda}(x/t)|&\lesssim t^{-1}.
\end{align*}
From \ref{Pointwise Elliptic Estimate} applied for $\partial_x\varphi$,
\begin{align*}
    \|\partial_x((1-\chi_{\lambda})\varphi_\lambda)\|_{L_x^{\infty}}&\lesssim \frac{1}{t}\|\chi'_{\lambda}\varphi_\lambda\|_{L_x^{\infty}}+\|(1-\chi_{\lambda})\partial_x\varphi_\lambda\|_{L_x^{\infty}} \lesssim \frac{\lambda^{1/2}}{t}\|\varphi\|_X.
\end{align*}
The first two bounds immediately follow from \ref{Pointwise Elliptic Estimate}, and the $L^2$ elliptic estimate similarly follows from \cite[Proposition 5.1]{ITpax}.

For the bounds involving the difference quotient, from \ref{Pointwise Elliptic Estimate} applied for $\dq^y\varphi$, we have
\begin{align*}
\|(1-\chi_{\lambda})\dq^y\varphi_\lambda\|_{L_x^\infty}&\lesssim \frac{\lambda^{1/2}}{t}(\|L\dq^y\varphi_\lambda\|_{L_x^2}+\lambda^{-1}\|\dq^y\varphi_\lambda\|_{L_x^2})\\
&\lesssim \frac{\lambda^{1/2}}{t}(\|\dq^y(L\varphi_\lambda)\|_{L_x^2}+\|\varphi_\lambda(x+y)\|_{L_x^2}+\|\varphi_\lambda\|_{L_x^2})\\
&\lesssim \frac{\lambda^{1/2}}{t}(\|L\partial_x\varphi_\lambda\|_{L_x^2}+\|\varphi_\lambda\|_{L_x^2})\\
&\lesssim \frac{\lambda^{1/2}}{t}\|\varphi\|_X
\end{align*}
The other bound is proved similarly.
\end{proof}

\subsection{Wave packets}

We construct wave packets as follows. Given the dispersion relation $a(\xi)$, the group velocity $v$ satisfies
\[
v = a'(\xi) = -2-2\log|\xi|,
\]
so we denote
\[
\xi_v = -e^{-1-\frac{v}{2}}.
\]
Then we define the linear wave packet $\pax^v$ associated with velocity $v$ by
\[
\pax^v = a''(\xi_v)^{-\half} \chi(y) e^{it\phi(x/t)}, \qquad y = \frac{x - vt}{t^\half a''(\xi_v)^\half},
\]
where the phase $\phi$ is given by
\[
\phi(v) = v\xi_v - a(\xi_v),
\]
and $\chi$ is a unit bump function, such that $\int\chi(y)\,dy=1$.

We remark that we will typically apply frequency localization $\pax^v_\lambda = P_\lambda\mathbf{u}^v$ with $v \in J_\lambda$.

\

We observe that since
\[
\D_v (|\xi_v|^{\half} ) = -\frac14 |\xi_v|^{\half}, \qquad \D_v (a''(\xi_v)^{-\half} ) = -\frac14 a''(\xi_v)^{-\half},
\]
we may write
\begin{equation}\label{dpax}
\D_v\pax^v = - \tilde L \pax^v + \pax^{v,II} = t^{\half} a''(\xi_v)^{-\half} \pax^v + \pax^{v,II} 
\end{equation}
where
 \[ \tilde{L}=t(\partial_x-i\phi'(x/t))\]
and $\pax^{v,II}$ has a similar wave packet form. We also recall from \cite[Lemmas 4.4, 5.10]{ITpax} the sense in which $\pax^v$ is a good approximate solution:

\begin{lemma}\label{l:paxsoln}
The wave packet $\pax^v$ solves an equation of the form
\[
(i\D_t - A(D)) \pax^v = t^{-\frac32}(L \pax^{v, I} + \rpax^v)
\]
where $\pax^{v, I}, \rpax^v$ have wave packet form,
\[
\pax^{v, I} \approx a''(\xi_v)^{-\half} \pax^v, \qquad \rpax^v \approx \xi_v^{-1}a''(\xi_v)^{-\half} \pax^v.
\]
\end{lemma}

\

The asymptotic profile at frequency $\lambda$ is meaningful when the associated spatial region $tJ_\lambda$ dominates the wave packet scale at frequency $\lambda$:
\[
\delta x \approx t^\half a''(\lambda)^\half \lesssim |tJ_\lambda| \approx t \lambda a''(\lambda).
\]
This corresponds to 
\[
t \gtrsim \lambda^{-2} a''(\lambda)^{-1} \approx \lambda^{-1}.
\]
Accordingly we define
\[
\mathcal D = \{(t, v) \in \R^+ \times \R : v \in J_\lambda, \ t \gtrsim \lambda^{-1} \}.
\]

\subsection{Wave packet testing}

In this section we establish estimates on the asymptotic profile function
\[
\gamma^\lambda(t, v) := \langle \varphi, \mathbf{u}^v_\lambda \rangle_{L^2_x}=\langle \varphi_\lambda, \mathbf{u}^v \rangle_{L^2_x}.
\]
We will see that $\gamma^\lambda$ essentially has support $v\in J_\lambda$.

\

We will also use the following crude bounds involving the higher regularity of $\gamma^\lambda$:
\begin{lemma}\label{Gamma bounds}
We have
    \begin{align*}
    \|\chi_\lambda \D_v^n \gamma^\lambda\|_{L^\infty}&\lesssim  t^{\half}(1+ t^\half\lambda^\half)^n\|\varphi_{\lambda}\|_{L_x^{\infty}}, \\
      \|\chi_\lambda \D_v^n\gamma^\lambda\|_{L^2}&\lesssim  (t\lambda)^{\frac14}(1+t^\half\lambda^\half)^n\|\varphi_{\lambda}\|_{L_x^{2}},
    \end{align*}
    and 
    \begin{align*}
\|\chi_\lambda\partial_v \gamma^\lambda\|_{L^\infty}&\lesssim t^{\frac14}\lambda^{-\frac34}\|\varphi\|_X + t^\half \|\varphi_\lambda\|_{L_x^{\infty}}.
    \end{align*}
    
\end{lemma}

\begin{proof}
Using the second form of $ \D_v \pax^v$ in \eqref{dpax}, we have
\[
|\chi_\lambda \D_v \gamma^\lambda| = |\chi_\lambda \langle \varphi_\lambda, \D_v \pax^v \rangle| \lesssim t^{\half}(t^\half\lambda^\half + 1)\|\varphi_{\lambda}\|_{L_x^{\infty}}
\]
where the $t^\half$ loss in front arises from the $L^1$ norm of the wave packet. Higher derivatives are obtained similarly, along with the $L^2$ estimates.

For the last bound, we use the first form of $ \D_v \pax^v$ in \eqref{dpax}. The contribution from the wave packet $\pax^{v,II}$ is easily estimated as above. For the remaining bound, Lemma 2.3 from \cite{ITpax} implies that
    \begin{align*}
        |\langle \varphi_\lambda,\tilde L \pax^v\rangle| &\lesssim (t\lambda)^{\frac14}\|\tilde{L}\varphi_\lambda\|_{L_x^2}\lesssim t^{\frac14}\lambda^{-\frac34}\|\varphi_\lambda\|_X,
         \end{align*}
which finishes the proof.
\end{proof}

\subsubsection{Approximate profile}

We recall from \cite{ITpax} that $\gamma^\lambda$ provides a good approximation for the profile of $\varphi$. In our setting, we will also need to compare the profile with the differentiated flow $\D_x \varphi$. Define
\[
r^\lambda(t, x) = \chi_\lambda(x/t)\varphi_\lambda(t, x) - t^{-\half}\chi_\lambda(x/t)\gamma^\lambda(t, x/t)e^{-it\phi(x/t)}.
\]

\begin{lemma}\label{Error bounds 1}
Let $t \geq 1$. Then we have
\begin{equation*}
\begin{aligned}
\|\chi_\lambda(x/t) r^\lambda\|_{L^\infty_x} &\lesssim t^{-\frac34}\lambda^{-\frac14} \| \tilde L \varphi_\lambda\|_{L^2_x}, \\
\|\chi_\lambda(x/t) \D_v r^\lambda\|_{L^\infty_x} &\lesssim t^{\frac14}\lambda^{-\frac14} \| \tilde L \D_x\varphi_\lambda\|_{L^2_x} +  (1 + t^\half \lambda^\half)\|\varphi_\lambda \|_{L^\infty}.
\end{aligned}
\end{equation*}
\end{lemma}

\begin{proof}
The first estimate may be obtained from the proof of \cite[Proposition 4.7]{ITpax}. For the latter, we use the first representation in \eqref{dpax} to write
\begin{equation}\label{rexp}
e^{it\phi(v)} \D_v(\gamma(t, v) e^{-it(\phi(v)}) = t\langle \D_x\varphi_\lambda, \pax^v \rangle + \langle \varphi_\lambda, it(\phi'(\cdot/t) - \phi'(v)) \pax^v \rangle + \langle \varphi_\lambda, \pax^{v, II} \rangle.
\end{equation}
To address the first term, we see that we may apply the undifferentiated estimate with $\D_x \varphi_\lambda$ in place of $\varphi_\lambda$. Precisely, we may apply the first estimate on
\[
\D_x\varphi_\lambda(t, x) - t^{-\half} \langle \D_x\varphi_\lambda, \pax^{x/t} \rangle e^{-it\phi(x/t)}.
\]
We estimate the third term of \eqref{rexp} via
\begin{equation*}
t^{-\half} |\langle \varphi_\lambda, \pax^{v, II} \rangle| \lesssim \|\varphi_\lambda \|_{L^\infty}.
\end{equation*}
It remains to estimate the middle term,
\[
 t^{-\half}|\chi_\lambda(v)\langle \varphi_\lambda, it(\phi'(\cdot/t) - \phi'(v)) \pax^v \rangle | \lesssim |\phi''(\lambda)| \cdot t^\half a''(\lambda)^\half \cdot \|\varphi_\lambda \|_{L^\infty} \lesssim t^\half \lambda^{\half} \|\varphi_\lambda \|_{L^\infty}
\]
\end{proof}

We also observe that on the wave packet scale, we may replace $\gamma(t, v)$ with $\gamma(t, x/t)$ up to acceptable errors. Denote
\[
\beta^\lambda_v(t, x) = t^{-1/2}\chi_\lambda(x/t)(\gamma(t, v) - \gamma(t, x/t))e^{it\phi(x/t)},
\]

\begin{lemma}\label{Beta bounds} 
Let $v\in J_\lambda$, and $(t,v)\in\mathcal{D}$. Then, for every $y\neq 0$ and $x$ such that $\displaystyle|x-vt|\lesssim\delta x=t^{1/2}\lambda^{-1/2}$, we have the bound 
\begin{align*}
|\dq^y\beta_v|&\lesssim t^{-3/4}\lambda^{-1/4}\|\varphi\|_X
\end{align*}
\end{lemma}
\begin{proof}
We have
\[
\delta^y \beta_v = -t^{-1/2}\delta^y(\gamma(t, \cdot/t)) \chi_\lambda((x+y)/t)e^{it\phi((x + y)/t)} + t^{-1/2}(\gamma(t, v) - \gamma(t, x/t)) \delta^y (\chi_\lambda(\cdot/t)e^{it\phi(\cdot/t)}).
\]
The Mean Value Theorem ensures that
\[
|\delta^y(\gamma(t, \cdot/t))| \lesssim t^{-1} \|\D_v \gamma\|_{L^\infty},
\]
and that
\begin{align*}
|\dq^y\beta_v|&\lesssim t^{-1/2}(t^{-1}\|\partial_v\gamma\|_{L^\infty}+t^{-1/2}\lambda^{-1/2}\|\partial_v\gamma\|_{L^\infty}(t^{-1}+\lambda))\\
&\lesssim t^{-1}\|\partial_v\gamma\|_{L^\infty}(t^{-1/2}+\lambda^{1/2})\lesssim t^{-1}\lambda^{1/2}(\lambda^{-3/4}t^{1/4}\|\varphi\|_X +\|\varphi_\lambda\|_{L_x^{\infty}}t^{1/2})\\
&\lesssim t^{-3/4}\lambda^{-1/4}\|\varphi\|_X+t^{-1/2}\lambda^{1/2}\|\varphi_\lambda\|_{L_x^\infty}\lesssim t^{-3/4}\lambda^{-1/4}\|\varphi\|_X+t^{-1}\lambda^{-1/4}\|\varphi\|_X\\
&\lesssim t^{-3/4}\lambda^{-1/4}\|\varphi\|_X
\end{align*}
\end{proof}

\subsection{Bounds for $Q$}

Write, slightly abusing notation, 
\[
Q(\varphi) = Q(\varphi, \overline{\varphi}, \varphi) := \frac13 \int \sgn(y) \cdot |\dq^y\varphi|^2\dq^y\varphi\, dy.
\]

\begin{lemma}\label{Cubic difference bound}
For $\displaystyle 0<\delta\ll 1$, we have the difference estimates 
\begin{equation*}\begin{aligned}
\|Q(\varphi_1) - Q(\varphi_2)\|_{L_x^\infty + L_x^{1/\delta}} &\lesssim (\|\D_x (\varphi_1, \varphi_2)\|_{L_x^\infty}^2 +\|\D_x(\varphi_1,\varphi_2)\|
_{L_x^{\frac{1}{2\delta}}}\|(\varphi_1,\varphi_2)\|_{L_x^\infty}) \|\partial_x(\varphi_1-\varphi_2)\|_{L_x^{\infty}},
\end{aligned}\end{equation*}
\begin{equation*}\begin{aligned}
\|Q(\varphi_1) - Q(\varphi_2)\|_{L_x^2}&\lesssim \|\D_x (\varphi_1,\varphi_2)\|_{L_x^2}\|\D_x(\varphi_1,\varphi_2)\|_{L_x^\infty}\|\partial_x(\varphi_1-\varphi_2)\|_{L_x^\infty}\\
&+\||D_x|^{1-\delta}(\varphi_1,\varphi_2)\|_{L_x^2}\|(\varphi_1,\varphi_2)\|_{L_x^\infty}\|\partial_x(\varphi_1-\varphi_2)\|_{L_x^\infty} .
\end{aligned}\end{equation*}
\end{lemma}

\begin{proof}
Write
\[
Q(\varphi_1) - Q(\varphi_2) = \int_{|y|\leq 1} + \int_{|y|> 1}
\]
where the integrand may be written
\begin{align*}
\sgn(y)(|\dq^y\varphi_1|^2\dq^y\overline{\varphi_1}- &|\dq^y\varphi_2|^2\dq^y\overline{\varphi_2}) \\
&= \sgn(y)(\dq^y(\varphi_1-\varphi_2)(|\dq^y\varphi_1|^2+|\dq^y\varphi_2|^2)+\dq^y(\overline{\varphi_1}-\overline{\varphi_2})\dq^y\varphi_1\dq^y\varphi_2).
\end{align*}
The first integral contributes to the two estimates respectively,
\begin{equation*}\begin{aligned}
\bigg| \int_{|y|\leq 1} \bigg| \lesssim \|\D_x(\varphi_1,\varphi_2)\|_{L_x^\infty}^2\|\partial_x(\varphi_1-\varphi_2)\|_{L_x^{\infty}}
\end{aligned}\end{equation*}
and
\[
\bigg\| \int_{|y|\leq 1}\bigg\|_{L^2} \lesssim \|\D_x(\varphi_1,\varphi_2)\|_{L_x^2}\|\D_x(\varphi_1,\varphi_2)\|_{L_x^\infty} \|\partial_x(\varphi_1-\varphi_2)\|_{L_x^{\infty}}.
\]

For the second, using Sobolev embedding,
\begin{equation*}\begin{aligned}
\bigg\| \int_{|y|> 1}\bigg\|_{L^{1/\delta}} &\lesssim \||D_x|^{1-\delta}(\varphi_1,\varphi_2)\|_{L_x^{1/\delta}}\|(\varphi_1,\varphi_2)\|_{L_x^\infty}\|\partial_x(\varphi_1-\varphi_2)\|_{L_x^{\infty}}\\
&\lesssim \|\partial_x(\varphi_1,\varphi_2)\|_{L_x^{1/(2\delta)}}\|(\varphi_1,\varphi_2)\|_{L_x^\infty}\|\partial_x(\varphi_1-\varphi_2)\|_{L_x^{\infty}}
\end{aligned}\end{equation*}
and 
\begin{equation*}\begin{aligned}
\bigg\| \int_{|y|> 1}\bigg\|_{L^{2}} &\lesssim\||D_x|^{1-\delta}(\varphi_1,\varphi_2)\|_{L_x^{2}}\|(\varphi_1,\varphi_2)\|_{L_x^\infty}  \|\partial_x(\varphi_1-\varphi_2)\|_{L_x^{\infty}}.
\end{aligned}\end{equation*}
\end{proof}

We will be considering separately the balanced and unbalanced components of $Q$. Precisely, we denote the diagonal set of frequencies by $\mathcal D$ and write
\begin{equation*}
\begin{aligned}
Q(\varphi, \varphi, \varphi) &= \sum_{(\lambda_1, \lambda_2, \lambda_3, \lambda) \in \mathcal D} Q(\varphi_{\lambda_1}, \varphi_{\lambda_2}, \varphi_{\lambda_3}) + \sum_{(\lambda_1, \lambda_2, \lambda_3, \lambda) \notin \mathcal D} Q(\varphi_{\lambda_1}, \varphi_{\lambda_2}, \varphi_{\lambda_3}) \\
&= Q^{bal}(\varphi, \varphi, \varphi) + Q^{unbal}(\varphi, \varphi, \varphi) = Q^{bal}(\varphi) + Q^{unbal}(\varphi).
\end{aligned}
\end{equation*}

\

The unbalanced portion of $Q$ satisfies the better bound as follows: 

\begin{lemma}\label{Unbalanced cubic estimate}
$Q^{unbal}$ satisfies the bounds
\[
\|\chi_\lambda^1\partial_xP_\lambda Q^{unbal}(\varphi)\|_{L_x^\infty} \lesssim \lambda^{-1/4}\frac{\|\varphi\|^3_X}{t^2}
\]
and
\[
\|\chi_\lambda^1\partial_xP_\lambda Q^{unbal}(\varphi)\|_{L_x^2} \lesssim \lambda^{-1/4}\frac{\|\varphi\|^3_X}{t^{3/2}},
\]
where $\chi^1_{\lambda}$ is a cut-off widening $\chi_{\lambda}$.
\end{lemma}

\begin{proof}
We shall denote
\begin{align*}
    I_{\lambda_1,\lambda_2,\lambda_3}&=\int_{\mathbb{R}}\sgn(y)\dq^y\varphi_{\lambda_1}\dq^y\varphi_{\lambda_2}\dq^y\varphi_{\lambda_3}\,dy
\end{align*}
and consider two cases in the frequency sum for $\partial_x P_\lambda Q^{unbal}$. 

First we consider the case in which we have two low separated frequencies. We assume without loss of generality that $\lambda_3=\lambda$ and $\lambda_1<\lambda_2\ll \lambda$. In this case, the elliptic estimates will be applied for the factor $\varphi_{\lambda_1}$. Precisely, from Lemma \ref{Trilinear integral estimate} and estimates \ref{Pointwise Vector Field Bound 2}, \ref{Pointwise Elliptic Estimate}, and \ref{L2 Elliptic Estimate}, we get that
\begin{align*}
    \left\|\chi^1_{\lambda}I_{\lambda_1,\lambda_2,\lambda_3}\right\|_{L_x^{\infty}}&\lesssim \lambda_1\frac{\lambda_1^{-1/2}}{t}\|\varphi\|_X(\lambda_2^{1-2\delta}+\lambda_2)\|\varphi_{\lambda_2}\|_{L_x^\infty}\lambda_3^{\delta}\|\varphi_{\lambda_3}\|_{L_x^\infty}\\
    &\lesssim \frac{\lambda_1^{1/2}}{t}\|\varphi\|_X\lambda_2^{1/4}(\lambda_2^{3/4-2\delta}+\lambda_2^{3/4})\|\varphi_{\lambda_2}\|_{L_x^\infty}\lambda^{-2}\lambda^{2+\delta}\|\varphi_{\lambda}\|_{L_x^\infty}\\
    &\lesssim \lambda_1^{1/2}\lambda_2^{1/4}\lambda^{-2}\frac{\|\varphi\|^3_X}{t^2}.
\end{align*}
By using dyadic summation in $\lambda_1$ and $\lambda_2$, we deduce that 

\begin{align*}
    \left\|\chi^1_{\lambda}\partial_x\sum_{\substack{
    \lambda_1<\lambda_2\ll \lambda}}I_{\lambda_1,\lambda_2,\lambda_3}\right\|_{L_x^{\infty}}&\lesssim \lambda^{-1/4}\frac{\|\varphi\|_X^3}{t^{2}}.
\end{align*}
Similarly, we deduce that 
\begin{align*}
    \left\|\chi^1_{\lambda}\partial_x\sum_{\substack{
    \lambda_1<\lambda_2\ll \lambda}}I_{\lambda_1,\lambda_2,\lambda_3}\right\|_{L_x^{2}}&\lesssim \lambda^{-1/4}\frac{\|\varphi\|_X^3}{t^{3/2}}
\end{align*}

We now analyze the situation in which $\lambda_1,\lambda_2\gtrsim \lambda$, and $\lambda_1$ and $\lambda_2$ are comparable and both separated from $\lambda$. Thus, we will be able to use $\lambda_1$ and $\lambda_2$ interchangeably.  We replace $\chi^1_{\lambda}$ by $\tilde{\chi}_{\lambda}$, which has double support, and equals $1$ on a comparably-sized neighbourhood of the support of $\chi^1_{\lambda}$. We write
\begin{align*}
    \chi^1_{\lambda}\partial_xP_\lambda=\chi^1_{\lambda}\partial_xP_\lambda\tilde{\chi}_{\lambda}+\chi^1_{\lambda}\partial_xP_\lambda(1-\tilde{\chi}_{\lambda}).
\end{align*}

For the first term, using Lemma \ref{Trilinear integral estimate}, along with estimates \ref{Pointwise Vector Field Bound 2}, \ref{Pointwise Elliptic Estimate}, \ref{L2 Elliptic Estimate}, we get the bounds 
\begin{align*}
    \left\|\chi^1_{\lambda}P_\lambda\tilde{\chi}_{\lambda}I_{\lambda_1,\lambda_2,\lambda_3}\right\|_{L_x^{\infty}}&\lesssim \lambda_2^{1/2+\delta}\frac{\lambda_3^{1-2\delta}+\lambda_3}{t}\|\varphi\|_X\|\varphi_{\lambda_2}\|_{L_x^\infty}\|\varphi_{\lambda_3}\|_{L_x^\infty}\\
    &\lesssim \lambda_2^{-5/4-\delta/2}\lambda_3^{\delta/2}\frac{\|\varphi\|_X}{t}(\lambda_3^{1-5\delta/2}+\lambda_3^{1-\delta/2})\|\varphi_{\lambda_3}\|_{L_x^\infty}\lambda_2^{7/4+3\delta/2}\|\varphi_{\lambda_2}\|_{L_x^\infty}\\
    &\lesssim\lambda_2^{-5/4-\delta/2}\lambda_3^{\delta/2}\frac{\|\varphi\|^3_X}{t^2}
\end{align*}
and
\begin{align*}
    \left\|\chi^1_{\lambda}P_\lambda\tilde{\chi}_{\lambda}I_{\lambda_1,\lambda_2,\lambda_3}\right\|_{L_x^{2}}&\lesssim \lambda_2^{\delta}\frac{\lambda_3^{1-2\delta}+\lambda_3}{t}\|\varphi\|_X\|\varphi_{\lambda_2}\|_{L_x^\infty}\|\varphi_{\lambda_3}\|_{L_x^\infty}\\
    &\lesssim \lambda_2^{-5/4-\delta/2}\lambda_3^{\delta/2}\frac{\|\varphi\|_X}{t}(\lambda_3^{1-5\delta/2}+\lambda_3^{1-\delta/2})\|\varphi_{\lambda_3}\|_{L_x^\infty}\lambda_2^{5/4+3\delta/2}\|\varphi_{\lambda_2}\|_{L_x^\infty}\\
    &\lesssim\lambda_2^{-5/4-\delta/2}\lambda_3^{\delta/2}\frac{\|\varphi\|^3_X}{t^2}.
\end{align*}
By using dyadic summation in $\lambda_1$, $\lambda_2$, and $\lambda_3$ (and by using the fact that $\lambda_1$ and $\lambda_2$ are close), we deduce the bound
\begin{align*}
   \left\|\chi^1_{\lambda}\partial_xP_\lambda\tilde{\chi}_{\lambda}\sum_{\substack{
    \lambda_3\lesssim \lambda_2,\lambda_1\simeq\lambda_2\gtrsim \lambda}}I_{\lambda_1,\lambda_2,\lambda_3}\right\|_{L_x^\infty\cap L_x^{2}}&\lesssim  \lambda^{-1/4}\frac{1}{t^2}\|\varphi\|^3_X. 
\end{align*}

We look at the second term. For every $N$, we know that
\begin{align*}
    \|\chi^1_{\lambda}\partial_xP_\lambda(1-\tilde{\chi}_{\lambda})\|_{L^2\rightarrow L^2}, \|\chi^1_{\lambda}\partial_xP_\lambda(1-\tilde{\chi}_{\lambda})\|_{L^\infty\rightarrow L^\infty}&\lesssim \frac{\lambda^{1-N}}{t^{N}}
\end{align*}
We take $N=\frac{3}{2}$. By carrying out a similar analysis as above, along with Lemma \ref{Trilinear integral estimate} and dyadic summation, we deduce that the contributions corresponding to these terms are also acceptable. 
\end{proof}
\begin{lemma}\label{Semiclassical computation}
We have
\begin{align*}\chi_\lambda((x/t))^3Q(e^{it\phi(x/t)})&=(\chi_\lambda(x/t))^3e^{it\phi(x/t)}q(\phi'(x/t))+h(\lambda,t),
\end{align*}
where for every $a\in(0,1)$
\begin{align*}
|h(\lambda,t)|&\lesssim \frac{\lambda^3}{t^{2-3a}}+\frac{\lambda^2}{t^{1-a}}+\frac{1}{t^{2a}}
\end{align*}
\end{lemma}
\begin{proof}
We write
    \begin{align*}
       e^{-it\phi(x/t)} \dq^ye^{it\phi(x/t)}&=\frac{e^{iy\phi'(x/t)}(e^{it/2\phi''(c_{x,y}/t)y^2/t^2}-1)}{y} + \frac{e^{iy\phi'(x/t)}-1}{y} =: a + b,
    \end{align*}
    where $c_{x,y}$ is between $x$ and $x+y$.
We now use the fact that $x/t$ belongs to the support of $\chi_\lambda$. We have
\begin{align*}
        |\chi_\lambda(x/t)||b| &\lesssim \lambda.
    \end{align*}
    Moreover, when $|y|\leq t^a$, $\displaystyle|c_{x,y}/t-x/t|\leq|y/t|\leq t^{a-1}$. This implies that $c_{x,y}/t$ belongs to the support of the enlarged cut-off $\chi_\lambda^1$, hence $\phi''(c_{x,y}/t)\simeq\lambda$. We note the bound
    \begin{align*}
        |\chi_\lambda(x/t)||a|&\lesssim |\chi_\lambda(x/t)||y/(2t)\phi''(c_{x,y}/t)|\left|\frac{e^{\pm i\phi''(c_{x,y}/t)y^2/(2t)}-1}{\phi''(c_{x,y}/t)y^2/(2t)}\right| \lesssim \lambda t^{a-1}
    \end{align*}
    Thus, we have the bounds
    \begin{equation}\label{ab-Bounds1}
    \begin{aligned}
|\chi_\lambda(x/t)||a|&\lesssim \lambda t^{a-1}\\
|\chi_\lambda(x/t)||b|&\lesssim \lambda
\end{aligned}
    \end{equation}
    We also note the cruder bounds
    \begin{equation}\label{ab-Bounds2}
    \begin{aligned}
|\chi_\lambda(x/t)||a|+
|\chi_\lambda(x/t)||b|&\lesssim \frac{1}{|y|}
\end{aligned}
    \end{equation}
    We write
    \begin{align*}
(\chi_\lambda(x/t))^3Q(e^{it\phi(x/t)})&=(\chi_\lambda(x/t))^3e^{it\phi(x/t)}\int \left|b\right|^2b\,dy\\
&\quad +(\chi_\lambda(x/t))^3e^{it\phi(x/t)}\int a^2\overline{a}+a^2\overline{b}+2|a|^2b+2a|b|^2+b^2\overline{a}\,dy:=T_1+T_2
    \end{align*}
We note that
\begin{align*}
    T_1&=(\chi_\lambda(x/t))^3e^{it\phi(x/t)}\int \left|b\right|^2b\,dy=(\chi_\lambda(x/t))^3e^{it\phi(x/t)}q(\phi'(x/t)),
\end{align*}
so we only need to analyze $T_2$.

We first bound the contribution over the region $|y|\leq t^a$, which we shall denote by $T_2^1$. We denote the contribution over the region $|y|>t^a$ by $T_2^2$. We have 
    \begin{align*}
T_2^1&=(\chi_\lambda(x/t))^3e^{it\phi(x/t)}\int_{|y|\leq t^a} a^2\overline{a}+a^2\overline{b}+2|a|^2b\,dy\\
&+(\chi_\lambda(x/t))^3e^{it\phi(x/t)}\int_{|y|\leq t^a} 2a|b|^2+b^2\overline{a}\,dy:=T_{2a}+T_{2b},
    \end{align*}
    \ref{ab-Bounds1} implies that 
    \begin{align*}
|T_{2a}|&\lesssim |\chi_{\lambda}(x/t)|^3\int_{|y|\leq t^a}|a|^3+|a|^2|b|\,dy\lesssim \int_{|y|\leq t^a}\frac{\lambda^2}{t^{2-2a}}\left(\frac{\lambda}{t^{1-a}}+\lambda\right)\lesssim \int_{|y|\leq t^a}\frac{\lambda^3}{t^{2-2a}}\,dy\\
&\lesssim t^a\frac{\lambda^3}{t^{2-2a}}\lesssim \frac{\lambda^3}{t^{2-3a}}
    \end{align*}
    
  \ref{ab-Bounds1} and \ref{ab-Bounds2} imply the bound
  \begin{align*}
|\chi_{\lambda}(x/t)|^3|b|^2|a|&=|\chi_{\lambda}(x/t)|^3\left|b\right|^2|y/(2t)\phi''(c_{x,y}/t)|\left|\frac{e^{\pm i\phi''(c_{x,y}/t)y^2/(2t)}-1}{\phi''(c_{x,y}/t)y^2/(2t)}\right|\\
&\lesssim\lambda\frac{1}{|y|}|\chi_{\lambda}(x/t)||y/(2t)\phi''(c_{x,y}/t)|\lesssim \frac{\lambda^2}{t}
  \end{align*}
    It follows that $T_{2b}$ satisfies the bound 
    \begin{align*}
|T_{2b}|&\lesssim |\chi_\lambda(x/t)|^3\int_{|y|\leq t^a}\left|b\right|^2|a|\,dy\lesssim\int_{|y|\leq t^a}\frac{\lambda^2}{t}\,dy\lesssim\frac{\lambda^2}{t^{1-a}}
    \end{align*}

For $T_2^2$, \ref{ab-Bounds2} implies that
  \begin{align*}
|T_2^2|\lesssim \int_{|y|>t^a}\frac{1}{|y|^3}\,dy\lesssim \frac{1}{t^{2a}}.
  \end{align*}
\end{proof}
\

\subsection{The asymptotic equation for $\gamma$}

Here we prove the following:

\begin{proposition}\label{Asymptotic equation}
Let $v\in J_\lambda$. Under the assumption $(t,v)\in\mathcal{D}$, we have
\begin{align*}
    \dot{\gamma}(t,v)=iq(\xi_v)\xi_vt^{-1}\gamma(t,v)|\gamma(t,v)|^2+f(t,v),
\end{align*}
where
\begin{align*}
    |f(t,v)| + \|f(t,v)\|_{L_v^2(J_{\lambda})} &\lesssim \lambda^{-1/4}t^{-6/5+C\eps^2}\eps.
\end{align*}
\end{proposition}
\begin{proof} 
We have
\begin{align*}
  \dot{\gamma}(t,v)=\left\langle \dot{\varphi},\mathbf{u}^v_\lambda\right\rangle+\left\langle \varphi,\dot{\mathbf{u}^v_\lambda}\right\rangle  &=\left\langle P_\lambda A_\varphi \varphi,\pax^v\right\rangle + i\langle \varphi_\lambda,(i\D_t - A(D))\mathbf{u}^v\rangle:=I_1+I_2.
\end{align*}

We first analyze $I_2$. We use Lemma~\ref{l:paxsoln} to write
\[
(i\D_t - A(D)) \pax^v = t^{-\frac32}(L \pax^{v, I} + \rpax^v)
\]
\begin{align*}
   \left|\langle \varphi_\lambda,(i\D_t - A(D))\mathbf{u}^v\rangle\right|&\lesssim t^{-\frac32}(\| L\varphi_\lambda\|_{L_x^2} \cdot \lambda^{1/2}\lambda^{1/4}t^{1/4} + \|\varphi_\lambda\|_{L_x^2}  \cdot \lambda^{-1/2}\lambda^{1/4}t^{1/4})\\
   &\lesssim \lambda^{-1/4} t^{-5/4}\|\varphi\|_X\\
   &\lesssim \lambda^{-1/4}t^{-5/4}\eps t^{C\eps^2}
\end{align*}
and 
\begin{align*}
   \|\chi_\lambda \langle \varphi_\lambda,(i\D_t - A(D))\mathbf{u}^v\rangle\|_{L_v^2}&\lesssim t^{-3/2}\left(\|L\varphi_\lambda\|_{L_x^2}\lambda^{1/2}+\|\varphi_\lambda\|_{L_x^2}\lambda^{-1/2}\right)\\
   &\lesssim \lambda^{-1/4} t^{-5/4}t^{-1/4}\lambda^{-1/4}\|\varphi\|_X\\
   &\lesssim \lambda^{-1/4} t^{-5/4}\eps t^{C\eps^2}\lesssim \lambda^{-1/4} t^{-6/5}\eps t^{C\eps^2}
\end{align*}
(we have used the condition $\displaystyle (t,v)\in\mathcal{D}$.)

\

In the remaining part of this section we shall analyze the term $I_1$. We first exchange $F$ for its principal quadratic term, expanding
\begin{align*}
   F(\dq^y\varphi)-\frac{1}{2}(\dq^y\varphi)^2=\int_0^1\frac{(1-h)^2}{2}(\dq^y\varphi)^3F'''(h\dq^y\varphi)\,dh.
\end{align*}
From Moser's estimate (the nonlinear version, as well as the one for products), Lemma \ref{Trilinear integral estimate}, and Sobolev embedding, we get that 

\begin{align*}
   \lambda^{\frac{1}{4}}&\left\|P_\lambda \int (\dq^y\varphi)^3F'''(t\dq^y\varphi)\sdq^y\varphi_x \,dy\right\|_{L_x^{\infty}}\\
   &\lesssim \int\||D_x|^{1/2}((\dq^y\varphi)^3\sdq^y\varphi_xF'''(t\dq^y\varphi))\|_{L_x^{4}}\,dy\\
   &\lesssim \|\varphi_x\|_{L_x^\infty}\||D_x|^{\frac{1}{2}}\varphi_{x}\|_{L_x^{4}}(\||D_x|^{1-\delta}\varphi\|_{L_x^\infty}\|\varphi_x\|^2_{L_x^\infty}+\|\varphi_x\|^2_{L_x^\infty}\|\varphi_{xx}\|_{L_x^\infty})\\
   &+\|\varphi_x\|_{L_x^\infty}\||D_x|^{\frac{1}{2}}\varphi_{xx}\|_{L_x^{4}}
   (\|\varphi_x\|^3_{L_x^\infty}+\|\varphi\|_{L_x^\infty}\|\varphi_x\|_{L_x^\infty}\||D_x|^{1-\delta}\varphi\|_{L_x^\infty})\\
   &\lesssim \frac{1}{t^{5/4}}\eps^5\langle t\rangle^{C\eps^2}.
\end{align*}

We have also used Sobolev embedding and the classical Moser estimate, keeping in mind $F'''(0)=0$. Similarly,
\begin{align*}
   \left\|\int_{\mathbb{R}}P_\lambda((F(\dq^y\varphi)-\frac{1}{2}(\dq^y\varphi)^2)\sdq^y\varphi_x)\,dy\right\|_{L_x^{2}}&\lesssim 
   \lambda^{-1/4}\frac{1}{t^{3/2}}\eps^5\langle t\rangle^{C\eps^2}.
\end{align*}

By H\"older's inequality and Young's inequality respectively,
\begin{align*}
\left|\left\langle P_\lambda \int_{\mathbb{R}}\left(F(\dq^y\varphi)-\frac{1}{2}(\dq^y\varphi)^2\right)\sdq^y\varphi_x\,dy,\mathbf{\varphi}_v\right\rangle\right|&\lesssim \lambda^{-1/4}\frac{1}{t^{3/2}}\eps^5\langle t\rangle^{C\eps^2}, \\
\left\|\left\langle P_\lambda\int_{\mathbb{R}}\left(F(\dq^y\varphi)-\frac{1}{2}(\dq^y\varphi)^2\right)\sdq^y\varphi_x\,dy,\mathbf{\varphi}_v\right\rangle\right\|_{L_v^2(J_{\lambda})}&\lesssim \lambda^{-1/4}\frac{1}{t^{3/2}}\eps^5\langle t\rangle^{C\eps^2}.\end{align*}

\

We are left to estimate
\begin{align*}
\left\langle P_\lambda\int_{\mathbb{R}}(\dq^y\varphi)^2\sdq^y\varphi_x\,dy,\mathbf{u}^v\right\rangle&=\left\langle \partial_x P_\lambda Q^{\text{bal}}(\varphi),\mathbf{u}^v\right\rangle+\left\langle \chi_\lambda^1\partial_x P_\lambda Q^{\text{unbal}}(\varphi),\mathbf{u}^v\right\rangle\\
&\quad +\left\langle (1-\chi_\lambda^1)\partial_x P_\lambda Q^{\text{unbal}}(\varphi),\mathbf{u}^v\right\rangle,
\end{align*}
where $\chi^1_{\lambda}$ be a cut-off function enlarging $\chi_{\lambda}$. Due to the fact that $\displaystyle \mathbf{u}^v$ is supported in the region $\displaystyle \left|\frac{x}{t}-v\right|\lesssim \lambda^{-1/2}t^{-1/2}$, the condition $(t,v)\in\mathcal{D}$ will imply that the third term is identically zero, while Lemma \ref{Unbalanced cubic estimate} implies that the second term is an acceptable error. 
Thus, we only have to analyze
\begin{align*}
\left\langle \partial_x P_\lambda Q^{\text{bal}}(\varphi),\mathbf{u}^v\right\rangle =\left\langle \partial_x P_\lambda Q(\varphi_\lambda),\mathbf{u}^v\right\rangle.
\end{align*}
 Let $\chi^1$ be a cut-off function that is equal to $1$ on the support of the wave packet $\mathbf{u}_v$. Let $\tilde{\chi}$ be another cut-off function whose support is  slightly larger than the one of $\chi^1$.
 We write
 \begin{align*}
\left\langle \partial_x P_\lambda Q(\varphi_\lambda),\mathbf{u}^v\right\rangle&=\left\langle \partial_x P_\lambda \tilde{\chi}Q(\varphi_\lambda),\mathbf{u}^v\right\rangle+\left\langle \chi^1\partial_x P_\lambda(1-\tilde{\chi})Q(\varphi_\lambda),\mathbf{u}^v\right\rangle
 \end{align*}
 As in the proof of Lemma \ref{Unbalanced cubic estimate}, we note that the operator norm bounds
 \begin{align*}
\|\chi^1\partial_xP_\lambda(1-\tilde{\chi})\|_{L^\infty\rightarrow L^\infty}+\|\chi^1\partial_xP_\lambda(1-\tilde{\chi})\|_{L^2\rightarrow L^2}&\lesssim \lambda^{1-2N}t^{-N}
 \end{align*}
 for every $N$ imply that the second term is acceptable error. This leaves us with the first.
 
We first replace $\varphi_\lambda$ by $\chi_\lambda\varphi_\lambda$. From Lemma \ref{Cubic difference bound}, we have
\begin{align*}
|\langle \partial_x P_\lambda \tilde{\chi}&(Q(\varphi_\lambda)-Q(\chi_\lambda\varphi_\lambda)),\mathbf{u}^v\rangle|\\
&\lesssim \lambda\|(\partial_x(\chi_\lambda\varphi_\lambda),\partial_x\varphi_\lambda)\|_{L_x^\infty}^2\|\partial_x((1-\chi_\lambda)\varphi_\lambda)\|_{L_x^{\infty}}(\|\mathbf{u}^v\|_{L_x^1}+\|\mathbf{u}^v\|_{L_x^{\frac{1}{1-\delta}}})\\
&+\lambda\|(\partial_x(\chi_\lambda\varphi_\lambda),\partial_x\varphi_\lambda)\|
_{L_x^{\frac{1}{2\delta}}}\|(\chi_\lambda\varphi_\lambda,\varphi_\lambda)\|_{L_x^\infty} \|\partial_x((1-\chi_\lambda)\varphi_\lambda)\|_{L_x^{\infty}}(\|\mathbf{u}^v\|_{L_x^1}+\|\mathbf{u}^v\|_{L_x^{\frac{1}{1-\delta}}})
\end{align*}
By interpolation, along with Lemma \ref{Elliptic bounds for the derivative} and the condition $(t,v)\in\mathcal{D}$, it follows that the errors are acceptable.  The $L_x^2$-bound is similar.

\

We now denote
\[
\psi(t,x)=t^{-\half}\chi_\lambda(x/t)\gamma(t,x/t)e^{it\phi(x/t)}
\]
and replace $\chi_\lambda\varphi_\lambda$ by $\psi$. From Lemma \ref{Cubic difference bound}, we have
\begin{align*}
|\langle \partial_x P_\lambda \tilde{\chi}&(Q(\chi_\lambda\varphi_\lambda)-Q(\psi)),\mathbf{u}^v\rangle|\\
&\lesssim \lambda\|(\partial_x(\chi_\lambda\varphi_\lambda),\partial_x\psi)\|_{L_x^\infty}^2\|\partial_x(\chi_\lambda(x/t) r^\lambda)\|_{L_x^{\infty}}(\|\mathbf{u}^v\|_{L_x^1}+\|\mathbf{u}^v\|_{L_x^{\frac{1}{1-\delta}}})\\
&+\lambda\|(\partial_x(\chi_\lambda\varphi_\lambda),\partial_x\psi)\|
_{L_x^{\frac{1}{2\delta}}}\|(\chi_\lambda\varphi_\lambda,\psi)\|_{L_x^\infty} \|\partial_x(\chi_\lambda(x/t) r^\lambda)\|_{L_x^{\infty}}(\|\mathbf{u}^v\|_{L_x^1}+\|\mathbf{u}^v\|_{L_x^{\frac{1}{1-\delta}}})
\end{align*}
By interpolation, along with Lemmas \ref{Error bounds 1} and \ref{Gamma bounds}, and the condition $(t,v)\in\mathcal{D}$, it follows that the errors are acceptable.  The $L_x^2$-bound is similar.

\
We now denote
\[
\theta(t,x)=t^{-\half}\chi_\lambda(x/t)\gamma(t,v)e^{it\phi(x/t)}
\]
and replace $\psi$ by $\theta$. We evaluate
\begin{align*}
\left\langle \partial_x P_\lambda \tilde{\chi}(Q(\psi)-Q(\theta)),\mathbf{u}^v\right\rangle
\end{align*}
We have
\begin{align*}
\left|\tilde{\chi}(Q(\psi)-Q(\theta))\right|&\lesssim \left|\tilde{\chi}\left(\frac{x-vt}{\sqrt{|ta''(\xi_v)|}}\right)\right|\int(|\dq^y\psi|^2+|\dq^y\theta|^2)|\delta^y\beta^\lambda_v(x)|\,dy
\end{align*}
The support condition of $\tilde{\chi}$ implies that $x$ is in the region $|x-vt|\lesssim \delta x=t^{1/2}\lambda^{-1/2}$. From Lemma \ref{Beta bounds} we now get that
\begin{align*}
\left|\tilde{\chi}(Q(\psi)-Q(\theta))\right|\lesssim t^{-3/4}\lambda^{-1/4}\|\varphi\|_X\int(|\dq^y\psi|^2+|\dq^y\theta|^2)\,dy 
\end{align*}
Bernstein's inequality, and Lemma \ref{Trilinear integral estimate}, and Sobolev embedding, imply that
\begin{align*}
\left|\tilde{\chi}(Q(\psi)-Q(\theta))\right|&\lesssim t^{-3/4}\lambda^{3/4}\|\varphi\|_X(\|\psi_x\|^2_{L_x^\infty}+\|\theta_x\|^2_{L^\infty})\|\mathbf{u}^v\|_{L_x^1}\\
&+t^{-3/4}\lambda^{3/4}\|\varphi\|_X(\|\psi\|_{L_x^\infty}\||D_x|^{1-\delta}\psi\|_{L_x^{\frac{1}{\delta}}}+\|\theta\|_{L_x^\infty}\||D_x|^{1-\delta}\theta\|_{L_x^{\frac{1}{\delta}}})\|\mathbf{u}^v\|_{L_x^{\frac{1}{1-\delta}}}\\
&\lesssim t^{-3/4}\lambda^{3/4}\|\varphi\|_X(\|\psi_x\|^2_{L_x^\infty}+\|\theta_x\|^2_{L^\infty})\|\mathbf{u}^v\|_{L_x^1}\\
&+t^{-3/4}\lambda^{3/4}\|\varphi\|_X(\|\psi\|_{L_x^\infty}\|\psi_x\|_{L_x^{\frac{1}{2\delta}}}+\|\theta\|_{L_x^\infty}\|\theta_x\|_{L_x^{\frac{1}{2\delta}}})\|\mathbf{u}^v\|_{L_x^{\frac{1}{1-\delta}}}
\end{align*}

From Lemma \ref{Gamma bounds} along with the condition $(t,v)\in\mathcal{D}$, it follows that this error is acceptable. The $L_x^2$-bound is similar.

\

We are left to analyze
\begin{align*}
t^{-3/2}\gamma(t,v)|\gamma(t,v)|^2\left\langle \partial_x P_\lambda Q(\chi_\lambda e^{it\phi(x/t)}),\mathbf{u}^v\right\rangle.
\end{align*}
Since by Lemma~\ref{Gamma bounds},
\[
\|t^{-3/2}\gamma(t,v)|\gamma(t,v)|^2\|_{L_v^\infty(J_\lambda)} \lesssim  \|\varphi_\lambda\|^3_{L_x^\infty}, \qquad \|t^{-3/2}\gamma(t,v)|\gamma(t,v)|^2\|_{L_v^2(J_\lambda)} \lesssim t^{-1/2}\|\varphi_\lambda\|^2_{L_x^\infty}\|\varphi_\lambda\|_{L_x^2}
\]
it suffices to estimate
\[
|\langle \partial_x P_\lambda Q(\chi_\lambda e^{it\phi(x/t)}),\mathbf{u}^v\rangle - t^{\half}q(\xi_v)\xi_v(\chi_\lambda(v))^3| \lesssim \lambda^{-1/4}t^{3/10+C\epsilon^2}\epsilon.
\]

 We note that
    \begin{align*}
\dq^y(\chi_{\lambda}e^{\pm it\phi(x/t)})&=\chi_{\lambda}\dq^y\left(e^{\pm it\phi(x/t)}\right)+\dq^y(\chi_{\lambda})e^{\pm it\phi((x+y)/t)}.
\end{align*}
Lemma \ref{Trilinear integral estimate}, implies that for every $\delta>0$ we have
\begin{align*}
\left|\left\langle\partial_x P_\lambda\int \dq^y(\chi_\lambda)e^{it\phi((x+y)/t)}\dq^y(\chi_\lambda e^{-it\phi(x/t)})\dq^y(\chi_\lambda e^{it\phi(x/t)})\,dy,\mathbf{u}^v \right\rangle\right|&\lesssim \lambda(t^{\delta-1/2}\lambda +t^{-1/2}\lambda^2).
\end{align*}
The most problematic contribution is the one that arises from the first term. We have
\begin{align*}
\lambda^{9/4}t^{\delta-1/2}\|\varphi_\lambda\|^3_{L_x^\infty}\lesssim t^{\delta-2}\|\varphi\|^3_X\lesssim t^{-6/5}\epsilon^3t^{C\epsilon^2}
\end{align*}
The other term is analogous.

The $L_v^2$-bound is treated similarly, and so is the case in which one chooses the term $\displaystyle\dq^y(\chi_{\lambda})e^{-it\phi((x+y)/t)}$ in the expansion of $\displaystyle \dq^y(\chi_{\lambda}e^{-it\phi(x/t)})$. This leaves us with
\begin{align*}
\langle \partial_x P_\lambda\left(\chi_\lambda(x/t)^3Q(e^{it\phi(x/t)})\right),\mathbf{u}^v\rangle
\end{align*}
Lemma \ref{Semiclassical computation}, implies that we can replace the latter with
\begin{align*}
\langle \partial_x P_\lambda\left(\chi_\lambda(x/t)^3e^{it\phi(x/t)}(\phi'(x/t))^2q(1)\right),\mathbf{u}^v\rangle,
\end{align*}
 with error bounded by 
\begin{align*}
\lambda t^{1/2} \left(\frac{\lambda^3}{t^{2-3a}}+\frac{\lambda^2}{t^{1-a}}+\frac{1}{t^{2a}}\right)
\end{align*}
We note that one problematic contribution is the one arising from the last term. We have the bound
\begin{align*}
\frac{\lambda^{5/4}}{t^{2a-1/2}}\|\varphi_\lambda\|^3_{L_x^\infty}&\lesssim t^{-2a}\|\varphi\|^3_X
\end{align*}
By picking $a=\frac{3}{5}$, we deduce that this contribution is acceptable. The only other problematic contribution is the one arising from the first term, for which we bound
\begin{align*}
\frac{\lambda^{17/4}}{t^{3/2-3a}}\|\varphi_\lambda\|^3_{L_x^\infty}&\lesssim t^{3a-3}\|\varphi\|^3_X\lesssim t^{-6/5}\epsilon^3 t^{C\epsilon^2}
\end{align*}
The contribution arising from the second term can be immediately bounded by
\begin{align*}
\frac{\lambda^{13/4}}{t^{1/2-a}}\|\varphi_\lambda\|^3_{L_x^\infty}&\lesssim t^{a-2}\|\varphi\|^3_X\lesssim t^{-7/5}\epsilon^3 t^{C\epsilon^2}\lesssim t^{-6/5}\epsilon^3 t^{C\epsilon^2}
\end{align*}
The $L_v^2$-bound is similar.

  This means that we have to analyze
\begin{align*}
&q(1)\langle \partial_x (\chi_\lambda(x/t)^3(\phi'(x/t))^2e^{it\phi(x/t)}),\mathbf{u}^v_\lambda \rangle =q(1)\langle (\chi_\lambda(x/t)\phi'(x/t))^3e^{it\phi(x/t)},\mathbf{u}^v_\lambda \rangle\\
&\quad +q(1)t^{-1}\langle (3\chi_\lambda(x/t)^2\chi_\lambda'(x/t)(\phi'(x/t))^2+2\chi_\lambda(x/t)^3\phi'(x/t)\phi''(x/t)e^{it\phi(x/t)}),\mathbf{u}^v_\lambda \rangle,
\end{align*}
where the last contribution can be immediately shown to be an acceptable error by using the condition $(t,v)\in\mathcal{D}$. Further, we may replace $\pax^v_\lambda$ by $\pax^v$. To see this, from the proof of Lemma 5.8 in \cite{ITpax}, we have
\begin{align*}
    \left|P_{\neq\lambda}\mathbf{u}^v\right|&\lesssim \lambda^{1/2}(1+|y|)^{-1-\delta}t^{-1-\delta}\lambda^{-1-\delta}, \qquad y=(x - vt)|ta''(\xi_v)|^{-\half},
\end{align*}
and
\begin{align*}
|(\chi_\lambda(x/t)\phi'(x/t))^3e^{it\phi(x/t)}|&\lesssim \lambda^3.
\end{align*}
Thus,
\begin{align*}
\left|\left\langle (\chi_\lambda(x/t)\phi'(x/t))^3e^{it\phi(x/t)},P_{\neq\lambda}\mathbf{u}^v\right\rangle\right|&\lesssim \lambda^3\lambda^{-1/2-\delta}t^{-1-\delta}t^{1/2}\lambda^{-1/2}\lesssim t^{-1/2-\delta}\lambda^{2-\delta},
\end{align*}
which along with the condition $(t,v)\in\mathcal{D}$ shows that this is an acceptable error.

As $\mathbf{u}^v$ is supported in the region $\displaystyle\left|\frac{x}{t}-v\right|\lesssim t^{-1/2}\lambda^{-1/2}$, we can replace $x/t$ by $v$ in $\chi_\lambda(x/t)\phi'(x/t)$, with acceptable errors. As $\chi_\lambda(v)=1$, the remaining term is now
\begin{align*}
iq(1)(\chi_\lambda(v)\xi_v)^3\left\langle e^{it\phi(x/t)},\mathbf{u}^v\right\rangle &=t^{\half}iq(\xi_v)\xi_v,
\end{align*}
as desired.
\end{proof}
\subsection{Closing the bootstrap argument}
We recall that
\begin{align*}
  \|\varphi_\lambda\|_{L_x^{\infty}}
  &\lesssim \frac{1}{\sqrt{t}}\lambda^{-(3/4-\delta-\delta_1)}\|\varphi\|_X\lesssim \frac{1}{\sqrt{t}}\lambda^{-(3/4-\delta-\delta_1)}\eps t^{C\eps^2},
\end{align*}
when $\lambda\leq 1$
and 
\begin{align*}
  \|\varphi_\lambda\|_{L_x^{\infty}}
  \lesssim \frac{1}{\sqrt{t}}\lambda^{-(2+3\delta/2)}\|\varphi\|_X\lesssim\frac{1}{\sqrt{t}}\lambda^{-(2+3\delta/2)}\epsilon t^{C\epsilon^2},
\end{align*}
when $\lambda>1$.

Thus, if $\displaystyle t\lesssim \lambda^{N}$ when $\lambda>1$, and if $\displaystyle t\lesssim \lambda^{-N}$ when $\lambda\leq 1$, where $N$ can be chosen arbitrarily, we get the desired bounds. We are left to analyze $\displaystyle t\gtrsim \lambda^{N}$ when $\lambda>1$, and $\displaystyle t\gtrsim \lambda^{-N}$ when $\lambda\leq 1$.

We recall the following bounds in the elliptic region:
\begin{align*}
    \||D_x|^{3/4-\delta}((1-\chi_{\lambda})\varphi_\lambda(x))\|_{L_x^{\infty}}&\lesssim \frac{\lambda^{1/4-\delta}}{t}\|\varphi\|_X\lesssim\frac{\lambda^{1/4-\delta}}{t}\eps t^{C^2\eps^2}\\
    \||D_x|^{1+\delta}\partial_x((1-\chi_{\lambda})\varphi_\lambda(x))\|_{L_x^{\infty}}&\lesssim \frac{\lambda^{3/2+\delta}}{t}\|\varphi\|_X\lesssim\frac{\lambda^{3/2+\delta}}{t}\eps t^{C^2\eps^2}
\end{align*}
which gives the desired bounds when $t\gtrsim \lambda^{N}$ ($\lambda>1$), and $t\gtrsim \lambda^{-N}$ ($\lambda\leq 1$).
We still have to bound $\displaystyle \chi_{\lambda}\varphi_\lambda$.
We recall that, if $x/t\in J_{\lambda}$, and $\displaystyle r(t,x)=\chi_{\lambda}\varphi_\lambda(t,x)-\frac{1}{\sqrt{t}}\chi_{\lambda}\gamma(t,x/t)e^{it\phi(x/t)}$,
\begin{align*}
   t^{1/2}\|r^\lambda\|_{L_x^{\infty}}&\lesssim t^{-1/4}\lambda^{-5/4}\eps t^{C\eps^2} 
\end{align*}

We note that
\begin{align*}
  t^{-1/4}\lambda^{-5/4}\eps t^{C\eps^2}&\lesssim \lambda^{-(3/4-\delta-\delta_1)}\eps   
\end{align*}
when $\lambda\leq 1$, because this is equivalent to
\begin{align*}
    \lambda^{-(1/2+\delta+\delta_1)}\lesssim t^{1/4-C\eps^2}
\end{align*}
(this is true when $t\gtrsim \lambda^{-N}$),
and that
\begin{align*}
  t^{-1/4}\lambda^{-5/4}\eps t^{C\eps^2}&\lesssim \lambda^{-(2+3\delta/2)}\eps   
\end{align*}
when $\lambda>1$, because this is equivalent to
\begin{align*}
    \lambda^{3/4+3\delta/2}\lesssim t^{1/4-C\eps^2}
\end{align*}
(this is true when $t\gtrsim \lambda^{N}$).

This means that we only need the bounds
\begin{align*}
   |\gamma(t,v)|\lesssim\eps \lambda^{-(3/4-\delta-\delta_1)} 
\end{align*}
when $\lambda\leq 1$, and
\begin{align*}
    |\gamma(t,v)|\lesssim\eps \lambda^{-(2+3\delta/2)} 
\end{align*}
when $\lambda>1$.
By initializing at time $t=1$, up to which the bounds are known to be true from the energy estimates, and by using Proposition \ref{Asymptotic equation}, we reach the desired conclusion.

\section{Modified scattering}\label{s:scattering}
In this section we discuss the modified scattering behaviour of the global solutions constructed in Section \ref{s:gwp}. We begin by proving the conservation of mass for the solutions of \eqref{SQG}:
\begin{proposition}\label{Conservation of mass}
   For solutions $\varphi$ of \eqref{SQG}, $\|\varphi(t)\|^2_{L^2}$ is conserved in time.  
\end{proposition}
\begin{proof}
We have
\begin{align*}
  \frac{d}{dt}\|\varphi\|^2_{L_x^2}&=\int\varphi_t\cdot\varphi\,dx\\
  &=-2\int\int F(\dq^y\varphi)\sdq^y\varphi_x\,dy\cdot\varphi\,dx-2\int \varphi\cdot \log|D_x|\varphi_x\,dx\\
  &=-2\int\int F(\dq^y\varphi)\sdq^y\varphi_x\,dy\cdot\varphi\,dx:=-2I.
\end{align*}

We note that by the change of variables $(x, y) \mapsto (x + y, -y)$,
\begin{align*}
I&= -\int\int F(\dq^{y}\varphi)\sdq^{y}\varphi_x\cdot\varphi(x+y)\,dx\,dy.
\end{align*}
Thus,
\begin{align*}
-2I&=\int\int F(\dq^{y}\varphi)\sdq^{y}\varphi_x\cdot(\varphi(x+y) - \varphi(x))\,dx\,dy\\
&=\int|y|\int F(\dq^{y}\varphi)\dq^y\varphi\cdot \dq^{y}\varphi_x\,dx\,dy\\
&=\int|y|\int \partial_x(G(\dq^{y}\varphi))\,dx\,dy=0,
\end{align*}
where $\displaystyle G(x)=\frac{x^2}{2}-\sqrt{1+x^2}$.
\end{proof}
Recall the asymptotic equation
\begin{align*}
    \dot{\gamma}(t,v)&=iq(\xi_v)\xi_v t^{-1}\left|\gamma(t,v)\right|^2\gamma(t,v)+f(t,v),
\end{align*}

As $t\rightarrow\infty$, $\gamma(t,v)$ converges to the solution of the equation
\begin{align*}
    \dot{\tilde{\gamma}}(t,v)&=iq(\xi_v)\xi_vt^{-1}\tilde{\gamma}(t,v)|\tilde{\gamma}(t,v)|^2,
\end{align*}
whose solution is 
\begin{align*}
    \tilde{\gamma}(t,v)&=W(v)e^{iq(\xi_v)\xi_v\ln(t)|W(v)|^2}
\end{align*}
We can immediately see that
$W(v)$ is well-defined, as
$|W(v)|=|\tilde{\gamma}(t,v)|$, which is a constant, and
\[
W(v)=\lim_{\substack{s\rightarrow\infty}}\tilde{\gamma}(e^{2s\pi/(q(\xi_v)\xi_v|W(v)|^2)},v).
\]

\begin{corollary}\label{Asymptotic expansions}
Let $v\in J_\lambda$. Under the assumption $(t,v)\in\mathcal{D}$, we have the asymptotic expansions
\begin{equation}\label{Wdiff}
    \|\gamma(t,v) - W(v)e^{ i q(\xi_v)\xi_v\log t |W(v)|^2}\|_{L^2 \cap L^\infty(J_\lambda)} \lesssim \lambda^{-1/4}t^{-1/5+C^2\eps^2}\eps.
\end{equation}
\end{corollary}
\begin{proof}
This is an immediate consequence of Proposition \ref{Asymptotic equation}.
\end{proof}
\begin{proposition}
Under the assumption
\begin{align*}
    \|\varphi_0\|_{X}\lesssim\eps\ll 1,
\end{align*}
the asymptotic profile $W$ defined above satisfies 
\begin{align*}
\|W(v)\|_{L_v^2}+\|e^{-\frac{v}{2}}e^{|v|C_1\epsilon^2}|D_v|^{1-C_1\epsilon^2}W(v)\|_{L_v^2}\lesssim\eps.
\end{align*}
\end{proposition}
\begin{proof}
We fix $\lambda$, and let $t\gtrsim\max\{1,\lambda^{-1}\}:=t_\lambda$. From Corollary \ref{Asymptotic expansions} we know that
\begin{align*}
\|W(v)-e^{- i q(\xi_v)\xi_v\log t |\gamma(t,v)|^2}\gamma(t,v)\|_{L_v^2(J_{\lambda})}\lesssim \lambda^{-1/4}t^{-1/5+C^2\eps^2}\eps
\end{align*}
From the product and chain rules with Lemma \ref{Gamma bounds}, we have 
\begin{align*}
    \left\|\partial_v\left(e^{-i q(\xi_v)\xi_v\log t |\gamma(t,v)|^2}\gamma(t,v)\right)\right\|_{L_v^2(J_{\lambda})}&\lesssim \lambda^{-1}\log(t)\eps t^{C^2\eps^2}.
\end{align*}
In this case,
\begin{align*}
W(v)=O_{\dot{H}^1_v(J_\lambda)}(\lambda^{-1}\log(t)\eps t^{C^2\eps^2})+O_{L_v^2(J_\lambda)}(\lambda^{-1/4}t^{-1/5+C^2\eps^2}\eps), \qquad t \gtrsim t_\lambda.
\end{align*}
By interpolation  this will imply that
for $C_1$ large enough we have
\begin{align*}
\|W(v)\|_{\dot{H}_v^{1-C_1\epsilon^2}(J_\lambda)}&\lesssim \lambda^{\frac{C_1\epsilon^2}{4}-1}\epsilon, \qquad \lambda > 1
\end{align*}
respectively 
\begin{align*}
\|W(v)\|_{\dot{H}_v^{1-C_1\epsilon^2}(J_\lambda)}&\lesssim \lambda^{-\frac{C_1\epsilon^2}{4}-1}\epsilon, \qquad \lambda <1 
\end{align*}

By dyadic summation over $\lambda\geq 1$ and $\lambda\leq 1$,
\begin{align*}
\|e^{-\frac{v}{2}}e^{|v|C_1\epsilon^2}|D_v|^{1-C_1\epsilon^2}W(v)\|_{L_v^2}\lesssim\epsilon
\end{align*}
\end{proof}

\bibliography{bib-sqg}
\bibliographystyle{plain}

\end{document}